\documentclass[12pt]{amsart}
\usepackage{geometry} 
\usepackage{color}
\usepackage[pdftex,bookmarks,colorlinks,breaklinks]{hyperref}  
\definecolor{dullmagenta}{rgb}{0.4,0,0.4}   
\definecolor{darkblue}{rgb}{0,0,0.4}
\hypersetup{linkcolor=red,citecolor=blue,filecolor=dullmagenta,urlcolor=darkblue} 

\usepackage{amsmath,amssymb,amsthm}
\usepackage{tikz}
\usetikzlibrary{arrows,matrix,positioning}
\usepackage{framed}
\usepackage{stmaryrd} 

\newtheorem{thm}{Theorem}[section]
\newtheorem{prop}{Proposition}[section]

\newtheorem{cor}{Corollary}[section]
\newtheorem{lem}{Lemma}[section]
\newtheorem{definition}{Definition}[section]
\newtheorem{example}{Example}[section]
\newtheorem{remark}[thm]{Remark}

\newcommand{\pder}[2]{\ensuremath{ \frac{ \partial #1 }{\partial #2 }
  } }

\newcommand{\into}{\ensuremath{\hookrightarrow}}
\newcommand{\restr}[2]{ \ensuremath{ \left. #1 \right|_{#2} } }
\newcommand{\lb}{ \ensuremath{\langle}}
\newcommand{\rb}{ \ensuremath{\rangle}}
\newcommand{\inter}{\ensuremath{\mathcal{I}}}

\DeclareMathOperator{\ver}{ver}
\DeclareMathOperator{\hor}{hor}
\DeclareMathOperator{\SDiff}{SDiff}
\DeclareMathOperator{\Diff}{Diff}
\DeclareMathOperator{\SL}{SL}
\DeclareMathOperator{\SO}{SO}
\DeclareMathOperator{\ad}{ad}
\DeclareMathOperator{\Ad}{Ad}

\DeclareMathOperator{\Euler}{Euler}

\begin{document}


\title{On the coupling between an ideal fluid and immersed particles}


\author[hoj]{Henry O. Jacobs}
\author[tsr]{Tudor S. Ratiu}
\author[md]{Mathieu Desbrun}





\begin{abstract}
In this paper we present finite dimensional particle-based
models for fluids which respect a number of geometric properties of the Euler equations of motion.
Specifically, we use Lagrange-Poincar\'{e} reduction to
understand the coupling between a fluid and a set of
Lagrangian particles that are supposed to simulate it. We substitute the use of principal connections in 
\cite{CeMaRa2001} with vector field valued interpolations
from particle velocity data.
The consequence of writing evolution equations in terms of
interpolation is two-fold. First, it provides estimates on the
error incurred when interpolation is used to derive the
evolution of the system.  Second, this form of the equations
of motion can inspire a family of particle and hybrid
particle-spectral methods, where the error analysis is
``built-in''. We also discuss the influence of other
parameters attached to the particles, such as shape,
orientation, or higher-order deformations, and how they can
help with conservation of momenta in the sense of Kelvin's
circulation theorem.

\end{abstract}



\maketitle

\section{Introduction}
Particles are capable of carrying a variety of information such as position, shape, orientation, et cetera.
Moreover, when this data is described with finitely many numbers, we may consider including it as an input for a computer simulation.
Given this point of view on particles, we seek to understand, from a geometric point of view, how a set of Lagrangian particles can be used as a computational device to numerically simulate an ideal fluid.
We will explore this idea by applying Lagrange-Poincar\'{e} reduction to the exact equations of motion.
The horizontal Lagrange-Poincar\'e equation can be used to inspire a family of particle methods.
Specifically, given an ideal, homogeneous, inviscid, incompressible fluid on a Riemannian manifold $M$ with smooth boundary $\partial M$, oriented by
the associated Riemannian volume,  the configuration space may be described by the group of volume
preserving diffeomorphism, $\operatorname{SDiff}(M)$, and the exact equations of motion for an ideal fluid are the $L^2$-geodesic equations~\cite{Arnold1966}. Throughout the paper
we shall assume that there is a Hodge decomposition; for compact boundary less manifolds this is standard (see, e.g., \cite{Jost2011}), for compact manifolds with
boundary this holds in the case of $\partial$-manifolds (i.e., the
manifold is, in addition, complete as a metric space, see
\cite{Schwarz1995}), for non-compact manifolds this holds
in function spaces with enough decay at infinity (for 
$\mathbb{R}^n$ see, e.g., \cite{Cantor1975}, \cite{GoTr2006}, 
\cite{Troyanov2009} and if the manifold has boundary, see 
\cite{Schwarz1995}).

If $\odot$ is an $N$-tuple of distinct points in $M$, let 
\[
G_\odot := \{ \psi \in\operatorname{SDiff}(M) \mid
\psi(\odot) = \odot \}
\]
be  the isotropy subgroup of the natural action of $\operatorname{SDiff}(M)$ on $M$.
The particle relabeling symmetry of an ideal fluid allows us to project the equations of motion onto the quotient space $T \operatorname{SDiff}(M)/G_\odot$.
Upon choosing an interpolation method, i.e., a means of interpolating a smooth vector field between the particles (see figure \ref{fig:interpolation} on page \pageref{fig:interpolation}), we obtain an isomorphism to a direct sum of vector bundles, $TX \oplus \tilde{\mathfrak{g}}_\odot$.
Here, the base manifold, $X$, is the configuration space of point particles in $M$, and $TX$ is the tangent bundle of $X$ (the state or velocity phase space).
The second component, $\tilde{\mathfrak{g}}_\odot$, is a vector bundle over $X$ whose fiber over $x \in X$ is the infinite dimensional vector space of vector fields which vanish at the particle locations described by $x$.  More generally, we can consider reducing by the subgroup
\[
	G_\odot^{(k)} := \{ \psi \in \SDiff(M) \mid T^{(k)}_\odot \psi = T^{(k)}_\odot id \}.
\]
The resulting Lagrange-Poincar\'{e} equations occur on a direct sum $TX^{(k)} \oplus \tilde{\mathfrak{g}}_{\odot}^{(k)}$, where $X^{(k)}$ is the configuration manifold of a more sophisticated type of particle which carries extra data such as orientation and shape; see \S\ref{sec:higher_order}
for the definition of all these objects.  The equations on $TX^{(k)} \oplus \tilde{\mathfrak{g}}_{\odot}^{(k)}$ describe the coupling between a fluid and this sophisticated type of particle in terms of interpolation methods.  The dynamics on the $TX^{(k)}$ component suggests a new class of particle methods.

\subsection{Organization and main contributions}
To understand the intent of this paper, it helps to explain what we mean by an \emph{interpolation method}.
While a formal definition will be given in \S \ref{sec:interpolation}, the idea is fairly simple.
Given $N$ particles in $M$ equipped with various data (e.g., position, velocity, orientation, higher-order deformations, etc), an interpolation method is a rule which produces a vector field on $M$ that is consistent with this data (see Figure \ref{fig:interpolation}).
Using the concept of an interpolation method, this paper accomplishes the following:
\begin{enumerate}
  \item For each interpolation method, we construct an isomorphism between the quotient space $( T\SDiff(M) ) / G_{\odot}$ and the vector bundle $TX \oplus \tilde{\mathfrak{g}}_{\odot}$ (Proposition \ref{prop:isomorphism}).
  \item We derive equations of motion on $TX \oplus \tilde{\mathfrak{g}}_\odot$ for an arbitrary $G_\odot$-invariant Lagrangian on $\SDiff(M)$ (i.e., the Lagrange-Poincar\'e equations, Theorem \ref{thm:main1}).
  \item We generalize these constructions to higher-order interpolation methods.  The resulting equations describe a family of particle methods wherein the particles carry extra data such as orientation, shape, and higher order deformations (Theorem \ref{thm:particle_method} and Corollary \ref{cor:particle_method}).
  \item The numerical methods of Corollary \ref{cor:particle_method} exhibit a particle-centric analog of Kelvin's circulation theorem (Theorem \ref{thm:discrete_kelvin})                                                                                                                   conserved quantities of the particle method correspond to conserved quantities of the exact equations of motion on $\SDiff(M)$.
  \item We illustrate how first-order interpolation methods induce particle methods which are related to the vortex blob method.
\end{enumerate}

In particular, these goals are accomplished as follows.
In \S \ref{sec:preliminary} we establish our notation and
review the notion of a generalized connection (also called an
\emph{Ehresmann connection}~\cite{MaMoRa1990}) as described in
~\cite{KMS99}.
In\S\ref{sec:reduced_system}, we carry out the reduction
process by the isotropy subgroup of a finite set of 
particles for an ideal incompressible homogeneous inviscid
fluid.
In \S\ref{sec:higher_order} we discuss reduction by
higher-order isotropies which allows us to study
particles with orientation, shape, and other  attributes.
In certain circumstances, this additional information 
produces particle methods which exhibit conservation
laws found in the exact dynamics on $\SDiff(M)$.
Finally, in \S\ref{sec:particle_methods}, we formulate
a family of particle methods induced by an interpolation
method and discuss some implications for the error analysis
of these methods. We conclude that it is possible to
construct hybrid particle-spectral methods for fluids within
this family.  Moreover, we show that the vortex blob
algorithm fits within this family of methods and that the
horizontal equations are a guide for corrections that allow
for the deformation of vortex blobs. We close
with \S\ref{sec:other_fluids}, where we summarize how to
extend these constructions to complex fluids, turbulence
models, and the template matching problems which
occur in medical imaging.

\subsection{Previous Work}
It was shown in~\cite{Arnold1966} that the Euler equations of
motion for an ideal, homogeneous, inviscid, incompressible
fluid on an oriented Riemannian manifold $M$ with smooth boundary are the spatial
(or Eulerian) representation of the geodesic equations on the
group of volume preserving diffeomorphisms,
$\operatorname{SDiff}(M)$.  This observation gave rise to a
new perspective on fluid mechanics which lead to many
developments, notably the proof of well posedness
\cite{EbinMarsden1970} and various
extensions ranging all the way to charged fluids,
magnetohydrodynamics, and even complex fluids with advected
parameters (see, e.g.,~\cite{Holm2002,GR2009}).  All
of these systems are Lagrangian on the tangent bundle of
groups of diffeomorphisms of a Riemannian manifold $M$.
Additionally, these theories utilize the
particle relabeling symmetry of the system to perform
Euler-Poincar\'{e} reduction and thus bring the dynamics to the
Lie algebra of this group~\cite[chapter 13]{MandS}.

As a result of this $\SDiff(M)$ symmetry We may consider 
reducing by subgroups of $\SDiff(M)$.
This would be a special case of Lagrange-Poincar\'{e}
reduction introduced and developed in~\cite{CeMaRa2001}.  In particular, we may consider
reducing by isotropy groups of a set of points in $M$.  Such
an approach is already mentioned in~\cite{MaWe1983} for vortex dynamics and in~\cite{MuDe2010} for the
purpose of landmark matching problems; see also the
references cited therein. However, to the best of our
knowledge, Lagrange-Poincar\'{e} reduction has not been
performed on such systems in the framework of~\cite{CeMaRa2001}.

\section{Preliminary Material}\label{sec:preliminary}
Before introducing our contributions, we review generalized
connections and volume-preserving diffeomorphisms and prove
a few important theorems.

\subsection{Generalized Connections}

In this section we introduce the notion of a generalized
connection, as presented in~\cite{KMS99}, and prove some
useful propositions for the purpose of this paper.

\begin{definition}
Let $\pi_E:E\to M$ be a vector bundle and
$\tau_E:TE\to E$ the tangent bundle of $E$.  The
\emph{vertical bundle} is the vector bundle $\pi_{V(E)}:V(E)
\to E$ where $V(E) := \mathrm{kernel}(T\pi_E)$ and
$\pi_{V(E)} := \restr{ \tau_E }{V(E) }$.
\end{definition}

The \emph{vertical lift} operator,
$v^\uparrow:E \oplus E \to V(E)$, is defined by
\[
v^\uparrow(v_m, w_m) :=
\left.\frac{d}{d\epsilon}\right|_{\epsilon=0}
\left(v_m + \epsilon w_m \right),
\]
for all $v _m, w _m\in E_m: = \pi_E^{-1}(m)$.
This establishes an isomorphism between the vector bundles
$\operatorname{proj}_1: E \oplus E \rightarrow E$ and
$\pi_{V(E)}: V(E)\rightarrow E$, where $\operatorname{proj}_1$
denotes the projection onto the
first summand. The \emph{fiber derivative},
$\pder{f}{e}:E \to E^*$, of a function
$f \in C^{\infty}(E)$ is defined by
\begin{equation}
\label{fiber_derivative}
\left\langle \pder{f}{e}(e) , e' \right\rangle :=
\left\langle df(e), v^\uparrow(e,e') \right\rangle \equiv
\left.\frac{d}{d\epsilon}\right|_{\epsilon=0}
f(e + \epsilon e').
\end{equation}
This notation suggests that we think of the fiber derivative $\pder{f}{e}$ as a type of partial derivative in which we only vary the `fiber-component'.  We shall also need a notion of partial differentiation with
respect to the base space $M$.  That is, we need
to understand what is meant by $\pder{f}{m}$
evaluated at $e \in E$, as an element of $T^*_mM$ over the base
point $m = \pi_E(e)$.  This is obtained by the use of a
covariant derivative.

\begin{definition} \label{def:generalized_connection}
Let $\pi_E:E \to M$ be a vector bundle.  A
\emph{horizontal bundle} is a subbundle, $H(E) \subset TE$,
such that $TE = H(E) \oplus V(E)$.  This defines projectors
$\hor:TE \to H(E)$ and $\ver:TE \to V(E)$.  The vertical
projector, $\ver$, is called a \emph{generalized connection}.
If $C^{\infty}(I;E)$ denotes the set of smooth curves in $E$ on
an open interval $I \subset \mathbb{R}$, then the
\emph{covariant derivative} induced by $H(E)$ (or $\ver$) is
the map $\frac{D}{Dt} : C^{\infty}(I;E) \to E$ given by
\begin{equation}
\label{covariant_derivative}
\frac{De}{Dt} = v_{\downarrow}\left( \ver \left( \frac{de}{dt}
\right) \right)
\end{equation}
Where $v_\downarrow:V(E) \to E :=
\operatorname{proj}_2 \circ (v^\uparrow)^{-1}$ and
$\mathrm{proj}_2: E \oplus E \to E$ is the projection onto
the second component.
\end{definition}

Naturally, the covariant derivative of a curve associated 
to a generalized connection on $E$ (see 
\eqref{covariant_derivative}) induces a covariant derivative
$\frac{D\alpha}{Dt}$ of a curve $\alpha(t)$ in the dual vector bundle $E^*$ by the relation
\[
	\frac{d}{dt} \left \langle \alpha , e \right \rangle = \left \langle \frac{D \alpha}{Dt} , e \right \rangle + \left \langle \alpha , \frac{De}{Dt} \right \rangle\,,
\]
where $e(t) \in E$ is an arbitrary curve.

Moreover, a generalized connection on $E$ induces the 
\emph{horizontal lift} operator
$h^\uparrow:E \oplus TM \to H(E)$ in the following way. 
Denote, as usual, by $V_e(E), H_e(E) \subset T_eE$ 
the fibers of the vertical and horizontal bundles, 
respectively, and note that $\dim V_e(E) = \dim E_{\pi_E(e)}$, 
$\dim H_e(E)= \dim T_eE - \dim V_e(E) = \dim E - 
\dim E_{\pi_E(e)} = \dim M$. Since $T_e\pi_E|_{H(E)}: H_e(E) 
\rightarrow T_{\pi_E(e)}M$ is injective, the dimension
count above implies that $T_e\pi_E|_{H(E)}$ is an isomorphism. Define the
\emph{horizontal lift} operator by 
\[
h^\uparrow(e, \dot{m}):= \left[\restr{(T_e\pi_E)}{H(E)}
\right]^{-1}(\dot{m}), \quad e \in E_m, \quad \dot{m} \in 
T_mM,
\]
i.e., $h^\uparrow( e , \dot{m})$ is the unique
horizontal vector in the fiber $T_eE$ such that
$T\pi_E( h^\uparrow(e,\dot{m})) = \dot{m}$.
Note that the horizontal lift operator induces a vector bundle 
isomorphism 
\begin{align*}
H(E) &\stackrel{\sim}\longrightarrow \pi_E^*(TM): = 
\left\{(e, \dot{m}) \in E \times TM \mid \pi_E(e) = 
\tau_M( \dot{m}) \right\}\\
v_e &\longmapsto\left(e, T_e \pi_E(v_e) \right)\\
h^\uparrow(e, \dot{m}) &\longmapsfrom (e, \dot{m})
\end{align*}
covering the identity on $E$ between the 
horizontal subbundle $H(E) \rightarrow E$ and the the 
pull-back bundle $\pi_E^*(TM) \rightarrow E$ of the tangent 
bundle $\tau_M:TM \rightarrow M$ by the vector bundle projection 
$\pi_E:E \rightarrow M$.

Finally, the choice of a generalized connection allows us to 
define the `partial derivative' of a function with respect to 
the base manifold.

\begin{definition} \label{def:partial_cov}
Let $\ver$ be a generalized connection on a vector bundle 
$\pi:E \to M$ and $V$ a vector space.  Let 
$f \in C^{\infty}(E, V)$.  Define the 
covariant derivative $\pder{f}{m}:E \to T^*M \otimes V$ of 
$f$ to be
the fiber bundle map covering the identity on $M$ given by
\begin{equation}
\label{covariant_derivative_function}
\left\langle \pder{f}{m}(e) , \delta m \right\rangle := \left
\langle df(e) , h^\uparrow(e,\delta m) \right\rangle \in V
\end{equation}
for all $e \in E$ and $\delta m \in T_mM$.
\end{definition}

By construction, this means that the total exterior
derivative $df\in\Omega^1(E ; V)$ acting on the velocity of a curve
$e(t) \in E$ over $m(t) \in M$ can be
written as the sum (see \eqref{fiber_derivative}, 
\eqref{covariant_derivative_function})
\begin{equation}
\left\langle df(e) , \frac{de}{dt} \right\rangle = 
\left\langle \pder{f}{m}(e) , \frac{dm}{dt} \right\rangle + 
\left\langle\pder{f}{e}(e) , \frac{D e}{Dt} \right\rangle. 
\label{eq:df=dmf+def}
\end{equation}

If $(M, g)$ is a Riemannian manifold we may use the 
Levi-Civita connection, defined on the vector bundle $TM$, 
to implement the previous constructions. Let $\Gamma_{ij}^k
: = g^{kh} \left(\frac{\partial g_{hi}}{\partial m^j}
+ \frac{\partial g_{hj}}{\partial m^i} - 
\frac{\partial g_{ij}}{\partial m^h} \right)$ be the 
Christoffel symbols of the Levi-Civita connection, where
$(m^1, \ldots, m^n)$ are local coordinates on $M$ and 
$g_{ij}: = g\left(\frac{\partial}{\partial m^i}, 
\frac{\partial}{\partial m^j}\right)$ the local components
of the metric tensor $g$. Since $\Gamma_{ij}^k = 
\Gamma_{ji}^{k}$ for all indices, it follows that this
connection is torsion free (and vice-versa). In natural
tangent bundle charts $(m^1, \dots, m^n, v^1, \dots, v^{n})$, 
the covariant derivative of a curve in $TM$ with respect to 
the Levi-Civita connection is given, locally, by 
$\frac{Dv^k}{Dt} = \frac{dv^k}{dt} + 
\Gamma_{ij}^{k} v^i \frac{dm}{dt}^j$. Therefore, the
  horizontal lift induced by the Levi-Civita connection is given locally by
\begin{align*}
h^\uparrow((m^i,v^j),(m^i,\delta m^j)) =
((m^i,v^j), (\delta m^l , -\Gamma_{ij}^{k} v^i \delta m^j)), 
\end{align*}
which implies that the covariant derivative \eqref{covariant_derivative_function} of a function $f \in C^{\infty}(TM, V)$ is given in local coordinates by
\begin{align}
	\left\langle \pder{f}{m}(v) , \delta m \right\rangle = \pder{f}{m^k} \delta m^k - \pder{f}{v^k} \Gamma_{ij}^{k}  v^i \delta m^j \,,
\label{eq:hor_lift_levi_civita}
\end{align}
where $v , \delta m \in T_mM$.
Note that the formulas for the horizontal lift $h^\uparrow$ and $\frac{\partial f}{\partial m}$ given above, are valid 
for any torsion free connection on $TM$, not just the 
Levi-Civita connection. This allows us to easily prove the 
following proposition which relates the notion of `torsion-free' to the exterior derivative.

\begin{prop} \label{prop:da(v,w)=da(w)v-da(v)w}
Let $\alpha \in \Omega^1(M, V)$.  Given a torsion free connection on
$TM$ (such as a Levi-Civita connection) we can consider the covariant derivative $\pder{ \alpha}{m}: TM \to T^{\ast}M$ by viewing $\alpha$ as a smooth function on $TM$.  This induces
the identity
\begin{equation}
d\alpha(m)( v , w ) = \left\langle \pder{\alpha}{m}(w) , v
\right\rangle - \left\langle
\pder{\alpha}{m}(v) , w \right\rangle\in V, \label{eq:da}
\end{equation}
where $v, w \in T_m M$ and `$d$' is the exterior derivative on $M$.
\end{prop}

\begin{proof}
This may be verified in a local coordinate chart, where
$\alpha(m) =\alpha_i (m) dm^i$.  Viewing $\alpha$
as a function on $TM$ with values in $V$ we find that in a local vector bundle chart $\alpha(m,v) = \alpha_i(m) v^i$.
Therefore, $\pder{\alpha}{v^k}(m) = \alpha_k(m)$. Thus, by invoking \eqref{eq:hor_lift_levi_civita} we arrive at
\[
\left\langle \pder{\alpha}{m}(m^i,w^j) ,
v^k \frac{\partial}{\partial m^k}\right\rangle =
\pder{\alpha_i}{m^j}(m) w^i v^j - \alpha_k(m)
\Gamma^k_{ij}(m) w^i v^j.
\]
Since $d\alpha(m)(v,w) = \left( \pder{\alpha_i}{m^j} -
\pder{\alpha_j}{m^i} \right) v^j w^i$, we can see that
\eqref{eq:da} follows from the torsion free property,
$\Gamma_{ij}^{k} = \Gamma_{ji}^{k}$.
\end{proof}

\begin{remark} \label{rmk:vector_valued_forms}
{\rm 
This proof is (formally) extendable to infinite dimensional manifolds by
noting the irrelevance of the local coordinate description
of $\Gamma_{ij}^k$.  Alternatively, a true coordinate-free proof is given in \cite[Proposition 2.4]{JaVa2012}. If $E$ is a trivial bundle, then this construction can be extended to $E$-valued one-forms
on $M$ by defining the exterior derivative on a tensor product by
\[
d( e \otimes \alpha ) := e \otimes d\alpha
\]
where $e$ is a section of a the vector bundle
$\pi:E \rightarrow M$ and $\alpha \in \Omega^1(M)$ is an
ordinary one-form on $M$. \quad $\lozenge$
}
\end{remark}

\subsection{Diffeomorphism Groups} \label{sec:diffeomorphism}
Let $M$ be a finite dimensional connected Riemannian manifold 
with metric $\left\langle \cdot , \cdot \right\rangle_M : TM 
\oplus TM \to \mathbb{R}$.
The set of volume preserving diffeomorphisms, $\SDiff(M)$, of 
$M$ is an infinite dimensional Fr\'echet Lie group.
A tangent vector $v_\varphi \in T_\varphi\SDiff(M)$ over a 
group element $\varphi \in \SDiff(M)$ is a map $v_\varphi: 
M \to TM$ such that $v_\varphi(m) \in T_{\varphi(m)}M$. Let
$\mathfrak{X}_{\mathrm{div}}(M)$ denote the vector space of
divergence free vector fields on $M$. The \textit{left} 
Lie algebra of $\SDiff(M)$, i.e., the Lie algebra
obtained by using the left
invariant vector fields on $\SDiff(M)$, is 
$(\mathfrak{X}_{\mathrm{div}}(M), 
- [\cdot , \cdot ]_{\rm Jacobi-Lie} )$, where
$[\cdot , \cdot ]_{\rm Jacobi-Lie}$ is the usual bracket of
vector fields on $M$ (locally, this means 
$[u, v]_{\rm Jacobi-Lie} = 
Dv\cdot u - Du\cdot v$). Therefore, $T_\varphi\SDiff(M) = \{u\circ \varphi \mid u \in \mathfrak{X}_{\mathrm{div}}(M)\}$.

Let $\odot = (\odot_1,\dots,\odot_N) \in M^N$ be an $N$-tuple of distinct points in $M$. We define the isotropy group
\begin{align}
G_\odot := \{\psi\in\SDiff(M) \mid \psi(\odot_k) = \odot_k ,\; k = 1, \ldots, N\}, \label{eq:G_odot}.
 \end{align}
 As a subgroup, $G_\odot$ acts on $\SDiff(M)$ by 
composition on the right, namely, $\SDiff(M)\ni \varphi
\mapsto \varphi \circ \psi \in \SDiff(M)$ for $\psi \in G_{\odot}$.
It is elementary to see that the quotient of $\SDiff(M)$ by the right action of $G_\odot$ is the manifold
\[
X := \{ (m_1,\dots,m_N) \in M^N \mid m_i \neq m_j\}
\]
with associated smooth projection $\pi_X(\varphi):= \left(
\varphi(\odot_1), \ldots, \varphi(\odot_N) \right)$.
In fact, $\pi_X:\SDiff(M) \rightarrow X$ is a right principal
$G_{\odot}$-bundle. Note that if $v_\varphi \in T_\varphi\SDiff(M)$, then $T_\varphi\pi_X(v_\varphi) = 
\left(v_\varphi(\odot_1), \ldots, v_\varphi(\odot_N)\right)\in 
TX$.

To streamline notation, we will write
$\varphi(\odot): = \left(\varphi(\odot_1), \ldots, \varphi(\odot_N) \right) \in X$. Similarly, if $v_\varphi \in 
T_\varphi\SDiff(M)$ (a vector field covering $\varphi$), 
we will write $v_\varphi(\odot): = \left(v_\varphi(\odot_1),
\ldots, v_\varphi(\odot_N)\right) \in TX$.

In \S \ref{sec:reduced_system} we will define a Lagrangian system on $T\SDiff(M)$ which is invariant under right multiplication by $G_\odot$.
We will then  implement Lagrange-Poincar\'{e} reduction along the lines of~\cite{CeMaRa2001} to obtain equations on the quotient space $( T\SDiff(M) ) / G_{\odot}$.
This task is less trivial than it may first sound because $(T\SDiff(M)) / G_\odot \neq T( \SDiff(M)/G_\odot)$.
In particular, $(T\SDiff(M))/G_\odot$ is a vector bundle over $X$ and the reduced equations describe how the motion of a finite set of particles couples to the fluid.

\begin{remark}{\rm 
We may consider applying the ideas contained in this paper to the full set of diffeomorphisms, $\Diff(M)$.  As our aim here is towards simulation of incompressible fluid equations, we have chosen to use $\SDiff(M)$.  However, extending the ideas to $\Diff(M)$ is not difficult, and such extensions may have some far reaching applications, such as in medical imaging \cite{Beg2005,Bruveris2011,Sommer2013}. } \quad $\lozenge$
\end{remark}

\subsection{Interpolation Methods as Generalized Connections}
\label{sec:interpolation}
Define the vector bundle $X$ given by
\[
\tilde{\mathfrak{g}}_\odot := \{ (x , \xi_x) \; \vert \; x \in
X ,\; \xi_x \in \mathfrak{g}_x \},
\]
where $\mathfrak{g}_x$ is the Lie algebra of the isotropy
group $G_x$ for each $x \in X$, i.e., the set of divergence
free vector fields which vanish at the $N$ particle locations 
described by $x \in X \subset M^N$
\footnote{The bundle $\tilde{\mathfrak{g}}_\odot$ is identical to the adjoint bundle $\frac{ \SDiff(M) \times 
\mathfrak{g}_\odot }{ G_\odot }$ when equipped with the fiberwise Lie bracket $ [(x,\xi_x) , (x,\eta_x) ] := (x , -[\xi_x,\eta_x]_{\mathrm{Jacobi-Lie}})$ and the projection $\tilde{\pi}(x,\xi_x) = x$.}.  Just as we have $\ad$-notation on a Lie algebra, we can utilize the $\ad$-notation on the Lie algebra bundle $\tilde{\mathfrak{g}}_\odot$.  For each $(x,\xi_x) \in \tilde{\mathfrak{g}}_\odot(x)$ we let $\ad_{(x,\xi_x)} : \tilde{\mathfrak{g}}_\odot(x) \to \tilde{\mathfrak{g}}_\odot(x)$ denote the linear endomorphism
\begin{equation}
\label{adjoint_bundle_bracket}
\ad_{(x,\xi_x)}( x, \eta_x) := [ (x,\xi_x) , (x, \eta_x)]
:= (x , -[\xi_x,\eta_x]_{\mathrm{Jacobi-Lie}}),
\end{equation}
where index ``Jacobi-Lie'' refers to the usual bracket of vector 
fields on a manifold. Additionally, we let $\ad^*_{(x,\xi_x)}: \tilde{\mathfrak{g}}_\odot^*(x) \to \tilde{\mathfrak{g}}_\odot^*(x)$ denote the dual endomorphism to $\ad_{(x,\xi_x)}$ on the dual vector-bundle $\tilde{\mathfrak{g}}_\odot^*$.

We will ultimately identify the quotient $(T\SDiff(M)) / G_\odot$ with $TX \oplus \tilde{\mathfrak{g}}_\odot$.  However, this identification is \emph{not} canonical.
It turns out that choosing such an identification boils down to the choice of an \emph{interpolation method}. 

\begin{definition}
\label{def_interpolation}
A $\mathfrak{X}_{\mathrm{div}}(M)$-valued one-form
$\inter \in \Omega^1(X ; \mathfrak{X}_{\mathrm{div}}(M) )$ on $X$ such that
$\inter(\dot{x}_1, \ldots, \dot{x}_N) (x_k) = \dot{x}_k$ for 
all $k=1, \ldots, N$, $\dot{x} \in TX$, and
$(x_1, \ldots, x_N) = \tau_X(\dot{x}_1, \ldots, \dot{x}_N)$ is called an
\emph{interpolation method}  {\rm (}see Figure \ref{fig:interpolation}{\rm ), where $\tau_X:TX \rightarrow X$ is the tangent bundle projection.
}\end{definition}

  \begin{figure}[h] 
     \centering
   \begin{tikzpicture}[thick,scale=0.7]
   	\draw[dashed] (-2.1,-2.1) rectangle (2.1,2.1);
	
	\def \px {{ -1.5 , -1.5 , 1.5 , 1.2, -0.4}}
	\def \py {{-1.5 , 1.0 , 1.5 , 0.4, 0.2 }}
	
	\foreach \i in {0,...,4}{
		\draw[fill=red] ( \px[\i]  , \py[\i] ) circle (0.1cm); 
		\draw[fill=red,xshift=6cm] ( \px[\i]  , \py[\i] ) circle (0.1cm); 
		\draw[->,draw=blue] (\px[\i],\py[\i]) -- ({\px[\i] - 0.2*(\py[\i] + 1)},{ 0.2*\px[\i] + \py[\i]});
		\draw[->,draw=blue,xshift=6cm] (\px[\i],\py[\i]) -- ({\px[\i] - 0.2*(\py[\i] + 1)}, {0.2*\px[\i] + \py[\i]});
	}
   	\draw[dashed,xshift=6cm](-2.1,-2.1) rectangle (2.1,2.1);
   	\foreach \x in {-2.0,-1.5,...,2.0}
		\foreach \y in {-2.0,-1.5,...,2.0}
			\draw[->,draw=blue,xshift=6cm] (\x,\y) -- ({\x-0.2*(\y+1)} ,{0.2*\x+\y});
	\draw[->,line width= 2pt] (1,0) .. controls (2,1) and (4,1)  .. node[above,scale=2]{$\inter$} (5,0) ;
   \end{tikzpicture}
     \caption{Schematic representation of an Interpolation Method}
     \label{fig:interpolation}
  \end{figure}
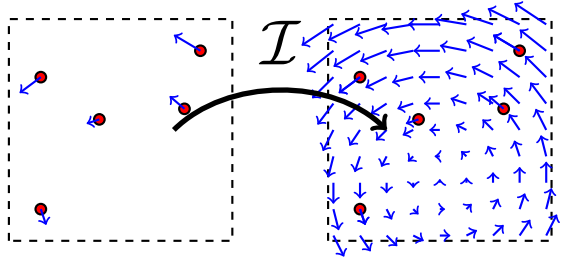

\begin{example} \label{ex:interpolation_method}
{\rm
In the case of Euler fluids, one could desire an interpolation method, $\inter \in \Omega^1( X; \mathfrak{X}_{\mathrm{div}}(M) )$, such that
\begin{align}
	\int_{M}{\lb \inter(\dot{x})(m),\xi_x(m) \rb d \mathrm{vol}(m)} = 0 \quad, \quad \text{for all } (x,\xi_x) \in \tilde{\mathfrak{g}}_\odot. \label{eq:mech_connection}
\end{align}
Such an interpolation method would correspond to choosing the \emph{mechanical connection}~\cite[\S 2.4]{MaMoRa1990}.
However, no such interpolation method exists.
For if $\dot{x} \in T_xX$ is non-zero, then the vector field $\inter(\dot{x})$ would need to vanish on the subset of $M$ complementary to the particles and simultaneously be non-zero when evaluated on at least one of the particles.
This does not define a smooth vector field and thus we can not expect to find an interpolation method which satisfies \eqref{eq:mech_connection}.
Nonetheless, regularized constructions can be made.
For example, if $M = \mathbb{R}^d$, and $\mathbb{I}: \mathfrak{X}(\mathbb{R}^d) \to \mathfrak{X}(\mathbb{R}^d)$ is a positive definite $\operatorname{SE}(d)$-invariant operator
($\operatorname{SE}(d)$ denotes the special Euclidean group
of proper rotations and translations in $\mathbb{R}^d$), we may define the inner product
\[
\left\langle u , v \right\rangle_{\mathbb{I}} = \int_{\mathbb{R}^d}{ u(m) \cdot \left[
    \mathbb{I} v \right] (m) d\mathrm{vol}(m)}.
\]
If the Green's function $G(\| x^i \|)$ associated to $\mathbb{I}$ is a continous function, then $\mathbb{I}$ naturally induces an interpolation method
\[
\inter( \dot{x} )(m) = \sum_{i}{ G(\|m - x_i\|) \dot{x}^i}.
\]
For example, if $\mathbb{I} = 1 - \alpha^2 \Delta$ for some 
$\alpha > 0$, then $G(\|x\|) = \exp( -\|x\|/\alpha)$.  This is not directly applicable in the context of ideal
incompressible fluids since this interpolation method does not produce
divergence free vector fields and the operator 
$1-\alpha^2\Delta$ is not naturally
identified with ideal fluids.  However, this construction is natural in the case of
the EPDiff equation with respect to the Lagrangian induced by the $H^1$-norm~\cite[part II]{HoScSt2009}. This construction has been generalized to inertia operators of the form
\[
	\mathbb{I} = \left( 1 - \frac{\alpha^2}{p} \Delta \right)^p \left(1 - \frac{1}{\epsilon^2} \mathrm{div} \circ \nabla \right)
\]
for $p \geq \frac{ \dim(M) + 3}{2}$ and $\epsilon > 0$ to products Greens functions of higher differentiability.  This would produce interpolation methods of higher differentiability as well~\cite{Mumford_Michor}. \quad $\lozenge$
}
\end{example}

\begin{remark} \label{rmk:connections} {\rm Even though we shall formulate all
constructions in terms of an interpolation method 
$\mathcal{I}$, we want to point out there is a bijective
correspondence between interpolation methods and
principal connections on the right $G_\odot$-bundle
$\pi: \SDiff(M)\ni \varphi \longmapsto \varphi(\odot) \in X$ 
(see, e.g., \cite{Bleecker1981}, \cite{KMS99}, or 
\cite{MaMoRa1990} adapted for right actions). We have $T_\varphi\pi(v_\varphi) = 
v_\varphi(\odot)$, for any $v_\varphi\in T_\varphi \SDiff(M)$ 
(a vector field covering $\varphi$, i.e., $v_\varphi:M \rightarrow TM$ satisfies $v_\varphi(m) \in 
T_{\varphi(m)}M$ for every $m \in M$).

Recall that if
$\pi:Q \to S$ is a right principal $G$-bundle and $\mathfrak{g}$ is the Lie algebra of $G$, then  a \emph{right equivariant principal connection} on $Q$ is a $\mathfrak{g}$-valued one-form
$A \in \Omega^1(Q;\mathfrak{g})$, satisfying the following properties:
\begin{enumerate}
\item For each $\xi \in \mathfrak{g}$, we have
$A(\xi_Q) = \xi$, where $\xi_Q \in \mathfrak{X}(Q)$, defined
by $\xi_Q(q):=\left.\frac{d}{dt}\right|_{t=0}
q \cdot \exp(t \xi)$ for any $q \in Q$, is the infinitesimal generator of $\xi$.
\item For each $g \in G$ and $v \in TQ$ we have
$A(v \cdot g) = \Ad_g^{-1}(A(v))$.
\end{enumerate}

If and interpolation method $\mathcal{I}: T
\rightarrow \mathfrak{X}_{\rm div}(M)$ is given, then 
\[
A (v_\varphi) := T \varphi^{-1}\circ v_\varphi
 - \varphi^*\big(\inter( v_\varphi(\odot) \big), \quad 
 \varphi \in \SDiff(M), \quad v_\varphi\in 
 T_\varphi \SDiff(M),
\]
defines a right principal connection one-form $A \in \Omega^1(\SDiff(M); \mathfrak{g}_\odot)$, as a direct
verification shows. Since $\left\{ v_\varphi \in 
T_ \varphi\SDiff(M) \mid v_\varphi = 
\mathcal{I}(v_\varphi( \odot))\circ \varphi \right\}$ is
the horizontal space defined by $A$, it is easily seen that
the associated horizontal lift has the expression
\[
h^\uparrow_\varphi(\dot{x}) = \inter(\dot{x}) \circ \varphi, \quad x=\varphi(\odot), \quad \dot{x} \in T_xX.
\]

Conversely, given a right principal 
connection one-form $A \in\Omega^1\left(\SDiff(M); 
\mathfrak{g}_\odot\right)$, it is easy to see that it defines
an interpolation method $\mathcal{I}:TX \rightarrow \mathfrak{X}_{\rm div}(M)$ by
\[
\mathcal{I}(T_\varphi \pi(v_\varphi)): = 
v_\varphi \circ \varphi^{-1} - \varphi_* A(v_\varphi).
\]
In addition, the principal connection induced by this
$\mathcal{I}$, returns the original $A$, which proves
the stated bijective correspondence between interpolation
methods and principal connections. \quad $\lozenge$
}
\end{remark}

Return to the general case and introduce the notation
$(x,\xi_x,\dot{x},\dot{\xi}_x)\in T\tilde{\mathfrak{g}}_\odot$ 
for tangent vectors at $(x,\xi_x)$ to 
$\tilde{\mathfrak{g}}_\odot$. Define the 
\textit{horizontal space} at $(x,\xi_x)$ induced by $\mathcal{I}$ to be the vector 
subspace of $T_{(x,\xi_x)}\tilde{\mathfrak{g}}_\odot$ 
consisting of vectors of the form
\[
\left(x,\xi_x, \dot{x}, 
[\xi_x,\inter \left(\dot{x} \right)]_{\rm Jacobi-Lie}\right)
\in T\tilde{\mathfrak{g}}_\odot.
\]
Thus, the horizontal and vertical projections are
\begin{align}
\hor( x,\xi_x,\dot{x},\dot{\xi}_x) &= 
\left(x, \xi_x, \dot{x}, 
[\xi_x, \mathcal{I}(\dot{x})]_{\rm Jacobi-Lie}\right) \label{eq:hor_proj} \\
\ver( x, \xi_x,\dot{x},\dot{\xi}_x) &= 
\left(x, \xi_x , 0 , \dot{\xi}_x -
[\xi_x,\inter(\dot{x})]_{\rm Jacobi-Lie} \right) 
\label{eq:ver_proj}.
\end{align}
\begin{prop}\label{prop:cov_der}
The covariant derivative induced by the horizontal projection \eqref{eq:hor_proj} is the $t$-curve in 
$\tilde{\mathfrak{g}}_{\odot}$ given by
\begin{equation}
\label{cov_der_xi_I}
\frac{D (x,\xi_x) }{Dt} = \left(x,\frac{d\xi_x}{dt} -
\left[\xi_x,\inter \left(
  \frac{dx}{dt} \right) \right]_{\rm Jacobi-Lie}\right).
\end{equation}
\end{prop}

\begin{proof}
Let $v =\frac{d}{dt}(x,\xi_x)$, so that the vertical and
horizontal projections are $\ver(v) = v - \hor(v)$,
and $\hor(v) = \left(x,\xi_x, \dot{x},
[\xi_x,\inter( \dot{x})]_{\rm Jacobi-Lie}\right)$, 
respectively.
Applying Definition \ref{def:generalized_connection}, we conclude
\begin{align*}
\frac{D(x,\xi_x)}{Dt} = v_\downarrow \left( \ver\left( \frac{d}{dt}(x,\xi_x)
\right) \right) = \left(x, \frac{d \xi_x}{dt} -
\left[\xi_x,\inter \left(
\frac{dx}{dt} \right) \right]_{\rm Jacobi-Lie} \right),
\end{align*}
as stated.
\end{proof}
  
\begin{prop} \label{prop:isomorphism}
  Given an interpolation method, $\inter:TX \to 
  \mathfrak{X}_{\mathrm{div}}(M)$, the map
  $\Psi_{\inter}:(T\SDiff(M))/G_\odot \to TX \oplus \tilde{\mathfrak{g}}_\odot$
given by
\[
\Psi_{\inter}( [ v_\varphi ] ) = \left(v_\varphi(\odot),
v_\varphi \circ \varphi^{-1} - \inter( v_\varphi(\odot) )\right)
\]
is an isomorphism of vector bundles, where $v_\varphi \in 
T_\varphi \SDiff(M)$, and $[v_\varphi]$
is the $G_\odot$-equivalence class of $v_\varphi$. 
\end{prop}

\begin{proof}
  We first must show that $\Psi_\inter$ is well defined on
  $(T\SDiff(M))/G_\odot$.  Let $\psi \in G_\odot$; note that
  \begin{align*}
    \Psi_\inter( [ v_\varphi \circ \psi ]
    ) &= (v_\varphi( \psi(\odot) ) , (v_\varphi \circ \psi) \circ
    (\varphi \circ \psi)^{-1} - \inter( v_\varphi(\psi(\odot)) ) ) \\
    &= (v_\varphi(\odot) , v_\varphi \circ \varphi^{-1} -
    \inter(v_\varphi(\odot) ) ) \\
    &= \Psi_\inter( [v_\varphi ] ).
  \end{align*}
It follows immediately from Definition \ref{def_interpolation} that $(v_ \varphi \circ \varphi - \mathcal{I}(v_\varphi(\odot)))(\varphi(\odot)) = 0$ which shows that 
$\Psi_{\inter}:(T\SDiff(M))/G_\odot \to TX \oplus \tilde{\mathfrak{g}}_\odot$ is well-defined.

Additionally, it is easy to check that $\Psi_\inter$ has the
  inverse $\Psi_\inter^{-1}( x,\dot{x} ,\xi_x) = 
  [ u \circ \varphi ]$ where $u = \xi_x + \inter(\dot{x})$ and $\varphi \in \SDiff(M)$ is an arbitrary
  diffeomorphism such that $x = \varphi( \odot)$.
A direct verification shows that $\Psi_{\inter}:T\SDiff(M)/G_\odot \to TX \oplus \tilde{\mathfrak{g}}_\odot$ is a vector
bundle map. 
\end{proof}

The importance of Proposition \ref{prop:isomorphism} is that it will allow us to transport objects defined on 
$T\SDiff(M) / G_\odot$ to objects defined on $TX \oplus \tilde{\mathfrak{g}}_\odot$ (the latter space being more intuitive).  In particular, we can express the equations of motion for an ideal fluid on $TX \oplus \tilde{\mathfrak{g}}_\odot$.  Moreover, while the $TX$ component represents the motion of the particles, the $\tilde{\mathfrak{g}}_\odot$ component represents the residual, which results from using $\inter(\dot{x})$ as an estimate of the velocity field of the fluid on $M$.  These observations will inspire particle methods for fluids wherein the error analysis is ``built-in''.

\section{Lagrangian Reduction and the Equations of Motion}
\label{sec:reduced_system}
  In this section we will use Proposition \ref{prop:isomorphism} to express the equations of motion for a $G_\odot$-invariant Lagrangian system on $\SDiff(M)$ on the space $TX \oplus \tilde{\mathfrak{g}}_\odot$. Concretely, we
carry out the program in \cite{CeMaRa2001} for this example.
However, instead of just quoting the relevant abstract formulas and applying them to our case, we prefer to derive
them by hand because we focus on interpolation methods
as opposed to connections (see remark  \ref{rmk:connections}) and want to keep the present paper self-contained.

Before we do this, however, we would like to motivate this task by considering the example of an ideal fluid.  To begin, the kinetic energy of an ideal fluid flowing on $M$, denoted
  $L_{\Euler}:T\SDiff(M) \to \mathbb{R}$, is given by
\begin{align}
L_{\Euler}(\varphi, \dot{\varphi} ) = \frac{\rho}{2} \int_{M}{ \|
  \dot{\varphi}(m) \|^2 d \mathrm{vol}(m) }, \label{eq:KE}
\end{align}
where $\rho$ denotes the density of the fluid, assumed to be
constant.
It was shown in~\cite{Arnold1966} that $L_{\Euler}$ is $\SDiff(M)$-invariant and that the resulting Euler-Poincar\'{e} (a.k.a. Euler-Arnold) equations are precisely Euler's equations for an ideal, inviscid, homogeneous, incompressible fluid
  \begin{align*}
    \pder{u}{t} + \nabla_u u = -\frac{\nabla p}{\rho}\,, \quad
\operatorname{div}u = 0, \quad u(m) \in T_m (\partial M) \;\;
\text{if} \;\; m \in \partial M,
  \end{align*}
for any oriented Riemannian manifold $M$ with smooth boundary
$\partial M$.
This result can be proven from a number of perspectives; we will focus here on the variational method.  The equations of motion on $T\SDiff(M)$ are the Euler-Lagrange equations
\[
\frac{D}{Dt} \left( \pder{L_{\Euler}}{\dot{\varphi}} \right) - \pder{L_{\Euler}}{\varphi} = 0.
\]
The integral curves of the Euler-Lagrange equations are known to extremize the action integral
\[
 S[ \varphi_t ] = \int_0^1 L_{\Euler}(\varphi_t, \dot{\varphi}_t) dt
\]

with respect to variations with fixed end-points at $t=0$ and $t=1$.  Conversely, curves which extremize $S$ must satisfy the Euler-Lagrange equations.  This variational principle is known as \emph{Hamilton's principle}, and is a central mechanism for deriving equations of motion for conservative mechanical systems~\cite[\S3]{FOM}.  In particular, one can derive reduced equations of motion on $\mathfrak{X}_{\mathrm{div}}(M) \equiv (T\SDiff(M) )/ \SDiff(M)$ by reducing Hamilton's principle.  This means that we must understand how a variation of a curve $\varphi_t \in \SDiff(M)$ induces variations of $u_t = \dot{\varphi}_t \circ \varphi_t^{-1} \in \mathfrak{X}_{\mathrm{div}}(M)$.  Understanding the relationship between variations of $\varphi_t$ and $u$ allows one to obtain a variational principle on $\mathfrak{X}_{\mathrm{div}}(M)$ which is equivalent to Hamilton's principle (up to a phase).  Moreover, this variational principle on $\mathfrak{X}_{\mathrm{div}}(M)$ induces equations of motion which are equivalent to Euler's equations for an ideal incompressible homogenous fluid.   In other words, Euler's fluid equations are the evolution equations obtained by reducing Euler-Lagrange equations on $T\SDiff(M)$ by a (right) $\SDiff(M)$ symmetry \cite{Arnold1966} (see also \cite[\S5.5]{FOM}, \cite{ArKh1992}, and \cite{HoScSt2009}).

In this section we shall do a reduction by $G_\odot \subset \SDiff(M)$ for an arbitrary $G_\odot$-invariant Lagrangian $L:T \SDiff(M) \to \mathbb{R}$ (such as $L_{\Euler}$) to derive an evolution equation on $TX \oplus \tilde{\mathfrak{g}}_\odot$.
  Of course, the resulting equations of motion yield the same dynamics as Euler's fluid equations when we choose the Lagrangian $L_{\Euler}$.
  However, despite producing the same dynamics, the use of the interpolation method heavily influences how one writes down the equations on $TX \oplus \tilde{\mathfrak{g}}_\odot$.
  As in the case of reduction by $\SDiff(M)$, we must first study how variations of curves in $\SDiff(M)$ lead to variations of curves in $TX \oplus \tilde{\mathfrak{g}}_\odot$.  This will allow us to define a variational principle on $TX \oplus \tilde{\mathfrak{g}}_\odot$ which yields the correct equations of motion.

\subsection{Covariant Variations}

Let $\varphi_t$ be a curve in $\SDiff(M)$.  A deformation of 
$\varphi_t$ is a two parameter family of diffeomorphisms, $\varphi_{\lambda , t}$, 
such that $\varphi_{0,t} = \varphi_t$.
We desire to measure how much the variation 
$\delta\varphi_t := \restr{ \pder{\varphi_{t,\lambda}}{\lambda} }{ \lambda =0}$ (a $t$-curve in $T\SDiff(M)$
covering $t \mapsto \varphi_t$) induces a variation in the quantity $(x,\xi_{x})_t :=
 (x(t),\dot{\varphi}_t \circ \varphi^{-1}_t - \inter(\dot{x}(t))) \in
\tilde{\mathfrak{g}}_\odot$ where $x(t) = \varphi_t(\odot)$ and
$\dot{\varphi}_t = \frac{d \varphi_t}{dt}$.  To do this, we invoke the covariant
derivative induced by $\inter$ in Proposition \ref{prop:cov_der} to define the \emph{covariant variation}
of a curve $(x,\xi_x)_t \in \tilde{\mathfrak{g}}_\odot$ with respect to a deformation
$(x,\xi_x)_{t, \lambda}$ as the quantity 
\begin{equation}
\label{def_covariant_variation}
\delta^{\inter}(x,\xi_x)_t := 
\restr{\frac{D(x, \xi_x)_{t, \lambda} }{ D \lambda }}{\lambda=0} \in \tilde{\mathfrak{g}}_\odot (x(t)).
\end{equation}
While this produces arbitrary variation 
of curves in $\tilde{\mathfrak{g}}_\odot$, we will only be
concerned with variations induced by variations of curves in
$\SDiff(M)$.  The following propositions describe the form of such variations.

\begin{prop}\label{prop:ver_var}
Let $\varphi_t$ be a curve in $\SDiff(M)$.  Then, a vertical
deformation of $\varphi_t$ is necessarily given by 
$\varphi_{t,\lambda} = \varphi_t \circ \psi_{t,\lambda}$ 
for a 2-parameter family $\psi_{t,\lambda} \in
G_\odot$ satisfying $\psi_{t,0} = id$. Set $x(t,\lambda):= 
\varphi_{t,\lambda}(\odot)$ and
\[
\xi_x(t,\lambda) : = \pder{\varphi_{t,\lambda} }{t} \circ 
\varphi_{t,\lambda}^{-1} - \mathcal{I}\left(\pder{x}{t}(t,
\lambda) \right),
\]
so that we have a 2-parameter family $(x,\xi_x)_{t,\lambda} 
\in \tilde{\mathfrak{g}}_\odot$. The covariant
variation of $(x,\xi_x)_{t,\lambda}$ with respect to 
$\lambda$ {\rm (}and with $t$-argument suppressed{\rm )} is
\[
\delta^{\mathcal{I}} (x,\xi_x) = \frac{D (x,\eta_x) }{Dt} - 
[(x,\xi_x),(x, \eta_x) ]
= \frac{D (x,\eta_x) }{Dt} + 
\left(x,[\xi_x,\eta_x ]_{\rm Jacobi-Lie}\right),
\]
where $\eta_x := (\varphi_t)_*\left(\restr{\pder{
    \psi_{t,\lambda} }{\lambda } }{\lambda=0} \right)\in \left(\tilde{\mathfrak{g}}_\odot\right)_x$. 
\end{prop}

\begin{proof}
If we let $\delta = \left. \pder{}{\lambda} \right|_{\lambda = 0 }$ then we can see that $\delta x = 0$ since $x = \varphi_t( \psi_{t,\lambda}( \odot) ) = \varphi_t(\odot)$ does not depend on $\lambda$.  Moreover, by Proposition \ref{prop:cov_der} we find 
	\begin{align*}
		\delta^{\mathcal{I}}( x, \xi_x) = \left( x , \delta \xi_x \right).
	\end{align*}
	
Next we calculate $\delta \xi_x$.  To do this, define the spatial velocity field $u = \dot{\varphi} \circ \varphi^{-1}$ so that $\xi_x = u - \mathcal{I}( \dot{x})$.   Then we find
\[
\delta \xi_x = \delta u - \delta ( \mathcal{I}(\dot{x})).
\]
However, as $x(t)$ does not depend on $\lambda$, we have 
$\delta ( \mathcal{I}( \dot{x}) ) = 0$.  Thus, by the Lin 
constraints derived in \ref{app:Lin} and the identity $\delta \varphi \circ \varphi^{-1} = \eta_x$, we have
\begin{align*}
\delta \xi_x &= \delta u \stackrel{\eqref{lin}}= 
\pder{ \eta_x }{t} - [ \eta_x , u ]_{\rm Jacobi-Lie} \\
&= \pder{\eta_x}{t} - 
[\eta_x, \mathcal{I}(\dot{x})]_{\rm Jacobi-Lie} - 
[\eta_x , \xi_x]_{\rm Jacobi-Lie}\,.
\end{align*}
Therefore,
\begin{align*}
\delta^ \mathcal{I}(x, \xi_x) &= (x, \delta\xi_x)
= \left(x, \pder{\eta_x}{t} - 
[\eta_x, \mathcal{I}(\dot{x})]_{\rm Jacobi-Lie} - 
[\eta_x , \xi_x]_{\rm Jacobi-Lie}\right)\\
& = \left(x, \pder{\eta_x}{t} - 
[\eta_x, \mathcal{I}(\dot{x})]_{\rm Jacobi-Lie}\right)
- \left(x,[\eta_x , \xi_x]_{\rm Jacobi-Lie}\right)\\
& = \frac{D}{ D t}(x, \eta_x) + \left[(x, \eta_x), 
( x, \xi_x) \right]
\end{align*}
by Proposition \ref{prop:cov_der} and \eqref{adjoint_bundle_bracket}.
\end{proof}
    
While the form of covariant variations induced by vertical variations is now clear, we must also consider how
$(x,\xi_x) \!= (x, \dot{\varphi}\circ  \varphi^{-1} \!-\! \inter(\dot{x})) \in \tilde{\mathfrak{g}}_\odot$ varies in response to variations of
$x = \varphi_t(\odot) \!\in\! X$.  Given a curve $\varphi_t \!\in\! \SDiff(M)$ we may take a deformation of
the curve $x(t) \!=\! \varphi_t(\odot) \!\in\! X$ given by $x(t,\lambda)$. Define $\delta x(t, \lambda): = 
\frac{\partial}{\partial\lambda} x(t, \lambda)\in T_{x(t, \lambda)}X$. Thus, $\mathcal{I}(\delta x(t, \lambda))\in 
\mathfrak{X}_{\rm div}(M)$. For each fixed $t$, $\lambda
\mapsto \mathcal{I}(\delta x(t, \lambda))$ is a $\lambda$-dependent family of divergence free vector fields on $M$.
Denote by ${\rm Fl}_\lambda^{\mathcal{I}(\delta x(t, \lambda))}$ the evolution operator of this $\lambda$-dependent
vector field which at $\lambda=0$ is the identity diffeomorphism on $M$. Define the horizontal deformation 
of $\varphi_t$ by
\begin{equation}
\label{eq:hor_var}
\varphi_{t,\lambda} := {\rm Fl}_\lambda^{\mathcal{I}(\delta x(t, \lambda))} \circ \varphi_t \in \SDiff(M).
\end{equation}
Therefore
\begin{equation}
\label{hor_phi_t}
\delta \varphi_t: = 
\left.\frac{\partial}{\partial\lambda}\right|_{\lambda=0}
\varphi_{t,\lambda} = \mathcal{I}(\delta x(t,0))\circ \varphi_t.
\end{equation}
In addition, by Definition \ref{def_interpolation}, 
for fixed $t$, $\lambda
\mapsto x_k(t, \lambda)$ is an integral curve of  
$\mathcal{I}(\delta x(t, \lambda))$ which at $\lambda=0$ 
passes through $x_k(t)$ and thus, by uniqueness of integral
curves, we have
\[
x_k(t, \lambda) = {\rm Fl}_\lambda^{\mathcal{I}(\delta x(t, \lambda))} (x_k(t)).
\]

Recall from Definition \ref{def_interpolation}, that an 
interpolation method is an element 
$\mathcal{I} \in\Omega ^1(X; \mathfrak{X}_{\rm div}(M))$, such
that $\mathcal{I}(\dot{x})(x) = \dot{x}$, for all 
$\dot{x} \in T_xX$. We want to define the exterior derivative
$d \mathcal{I}$ in a way which is consistent with the exterior
derivative on $\Omega^1(X)$.  We do this by considering the ordinary
one form $[\alpha] \circ \mathcal{I} \in \Omega^1(M)$ obtained for an element
of the dual space $[\alpha] \in \mathfrak{X}_{\mathrm{div}}(M)^\ast$.
At this point it is useful to recall what the (smooth) dual space of $\mathfrak{X}_{\rm div}(M)$ is.
 We recall (see \cite{MaWe1983}) that 
$\mathfrak{X}_{\mathrm{div}}(M)^\ast$ can be identified
with $\Omega^1(M)/d C^{\infty}(M)$ via the weakly 
non-degenerate pairing
\[
\left\langle [ \alpha], X \right\rangle: = \int_M \alpha(X)
\mu,
\]
where $\mu \in\Omega^n(M)$ is the the Riemannian volume form on $M$ (recall that $(M,g)$ is an oriented Riemannian 
manifold, possibly with boundary, for which the Hodge decomposition theorem holds). This follows from the Hodge decomposition theorem for one-forms. 

So, for any element
$[\alpha]\in \Omega^1(M)/dC^{\infty}(M) \cong
\mathfrak{X}_{\mathrm{div}}(M)^\ast$, the composition 
$[\alpha] \circ \mathcal{I}: TX \rightarrow \mathbb{R}$ is 
a usual one-form on $X$, given by
\[
\left([\alpha] \circ \mathcal{I}\right)(\delta x) : = 
\int_M \alpha\left(\mathcal{I}(\delta x)\right) \mu, 
\quad \text{for all} \quad \delta x \in T_xX.
\]
We define the exterior derivative $d \mathcal{I}$ of 
$\mathcal{I}$, to be the unique 
$\mathfrak{X}_{\mathrm{div}}(M)$-valued two-form on $X$ 
such that $[\alpha] \circ d \mathcal{I} = 
d([\alpha] \circ \mathcal{I})$ for any $\alpha \in 
\Omega^1(M)$.

Next, we turn to the definition of the covariant derivative
$\frac{\partial \mathcal{I}}{\partial x}$ of $\mathcal{I}$ relative to the base point in $X$.  
By Defintion \ref{def:partial_cov}, $\frac{\partial([\alpha] 
\circ \mathcal{I})}{\partial x}:TX\rightarrow T ^\ast X$ is
a fiber bundle map covering he identity. In complete analogy
with the definition of $d\mathcal{I}$, we define the map
$\pder{\mathcal{I}}{x}: TX \rightarrow T^\ast X \otimes 
\mathfrak{X}_{\mathrm{div}}(M)$ by
\[
[\alpha]\left(\left\langle \pder{\mathcal{I}}{x}(\dot{x}), 
\delta x \right\rangle \right) = \left \langle 
\pder{([\alpha] \circ \mathcal{I})}{x}(\dot{x}) , \delta x 
\right \rangle, \quad \text{for all} \quad \dot{x}, \delta x
\in T_xX.
\]
This is the standard definition of the covariant derivative
of a vector valued one-form on a Riemannian manifold (see, 
e.g., \cite[Definition 3.3.3]{CeMaRa2001}). 

More explicitly, we can locally write $\mathcal{I} = 
\varepsilon^i\otimes v_i $ for a finite number of local
vector field valued functions $v_1,\ldots,v_n\in C^{\infty}(X; 
\mathfrak{X}_{\mathrm{div}}(M))$ and a local basis of 
one-forms $\varepsilon^1, \dots, \varepsilon^{n} \in 
\Omega^1(X)$, where $n:= \dim X$. Then, a direct computation
using Definition \ref{def:partial_cov} and 
\eqref{eq:hor_lift_levi_civita} shows that
\[
\frac{\partial \mathcal{I}}{\partial x}\left(\dot{x}\right)
=\left(\varepsilon^i(x)\left(\dot{x}\right)\right) dv_i(x) + 
 \frac{\partial \varepsilon^i}{\partial x} \left(\dot{x}\right) \otimes v_i (x) \in T ^\ast_x X \otimes 
 \mathfrak{X}_{\rm div} (M),
\]
where $\frac{\partial \varepsilon^i}{\partial x}$ is the covariant derivative of $\epsilon^i$ in Definition 
\ref{def:partial_cov} (viewing $\varepsilon^i$ as a smooth
function on $TX$) and $dv_i$ is the exterior derivative
of $v_i \in C ^{\infty}(X;\mathfrak{X}_{\rm div} (M))$.

While $\frac{\partial \mathcal{I}}{ \partial x}$  provides a 
measure of how much the value of $\mathcal{I}$ varies as we 
change the base point of an element of $TX$, we would also 
like to describe how $\mathcal{I}$ varies with respect to 
changes in the velocity vector of $TX$ while holding the base 
point fixed, i.e., we need the fiber derivative of 
$\mathcal{I}$ (see \eqref{fiber_derivative}). Since 
$\mathcal{I}$ is linear in the velocity variable (by 
Definition \ref{def_interpolation}), it follows that the 
fiber derivative of $\mathcal{I}$ is $\mathcal{I}$ itself.

Putting  all these remarks together, it follows that along a 
smooth curve $\lambda \mapsto \dot{x}_\lambda \in TX$ 
covering the curve $\lambda \mapsto x_\lambda \in X$ with 
$\dot{x}_\lambda|_{\lambda=0} = \dot{x}$ and 
$\left.\frac{d}{d\lambda}\right|_{\lambda=0}x_{\lambda} = 
\delta x$, by equation \eqref{eq:df=dmf+def}, we have
\begin{equation}
\label{dI_derivative}
\left.\frac{d}{d\lambda}\right|_{\lambda=0}\mathcal{I}(\dot{x}) = 
\left\langle \pder{ \inter }{x}(\dot{x}) ,  \delta x \right\rangle  +  \inter \left( \left. \frac{D \dot{x} }{D\lambda} \right|_{\lambda = 0} \right) \in \mathfrak{X}_{\rm div}(M).
\end{equation}
This formula allows us to express horizontal variations in a particularly simple form.

\begin{prop}\label{prop:hor_var}
Let $\varphi_t$ be a curve in $\SDiff(M)$ and $x(t) =
\varphi_t(\odot)$.  Let $x_{t,\lambda}$ be a deformation of 
$x_t$ and  $\varphi_{t,\lambda}$ the resulting horizontal
deformation given by equation \eqref{eq:hor_var}.
Then the covariant variation of $(x,\xi_x)$, where 
$\xi_x := \dot{\varphi}\circ \varphi^{-1} - \inter(\dot{x})$, 
is given by
\[
\delta^{\inter} (x,\xi_x) = \widetilde{B}(\dot{x}, \delta x),
\]
where $\widetilde{B}$ is the 
$\tilde{\mathfrak{g}}_\odot$-valued two-form on $X$ given 
by the expression
\begin{align}
\widetilde{B}(\dot{x},\delta x) = 
(x, d\inter(\dot{x},\delta x) + 
[\inter(\dot{x}) , \inter(\delta x)]_{ \rm Jacobi-Lie}). 
\label{eq:curvature_tensor}
\end{align}
\end{prop}

Before we begin the proof we would like to point out that 
$\widetilde{B}$ is the \textit{reduced curvature tensor} of 
principal connection associated to $\mathcal{I}$ in remark
\ref{rmk:connections}.  This relationship bridges  the reductions 
being performed in this paper with those performed in 
\cite{CeMaRa2001}.

\begin{proof}
	Define the spatial velocity $u := \dot{\varphi} \circ \varphi^{-1} \in \mathfrak{X}_{\rm div}(M)$.  Then we see that $\xi_x = u - \mathcal{I}(\dot{x})$. Let us now calculate the derivative of $\xi_x$ with respect to $\lambda$ at $\lambda = 0$.  We find
	\[
		\delta \xi_x = \delta u - \left. \frac{ \partial }{\partial \lambda} \right|_{\lambda = 0} ( \mathcal{I}(\dot{x}))
	\]
	Upon noting that $\delta \varphi \circ \varphi^{-1} = \mathcal{I}(\delta x)$ we see by the Lin constraints derived in \ref{app:Lin} (see \eqref{lin}) that
	\[
		\delta u = \frac{\partial}{\partial t} ( \mathcal{I}(\delta x) ) - [ \mathcal{I}( \delta x) , u ]_{\rm Jacobi-Lie}
	\]
	By \eqref{dI_derivative} we find
	\begin{align*}
		\left. \frac{ \partial }{\partial \lambda} \right|_{\lambda = 0} ( \mathcal{I}(\dot{x})) &= \left \lb \pder{\mathcal{I}}{x}(\dot{x}) , \delta x \right \rb + \mathcal{I} \left( \frac{D\dot{x}}{D \lambda} \right) \\
		 \frac{\partial}{\partial t} ( \mathcal{I}(\delta x) ) &= \left \lb \pder{\mathcal{I}}{x}(\delta x) , \dot{x} \right \rb + \mathcal{I} \left( \frac{D\delta x}{D t} \right).
	\end{align*}
	We may now express $\delta \xi_x$ as
	\begin{align*}
		\delta \xi_x =& \underbrace{ \left \lb \pder{\mathcal{I}}{x}(\delta x) , \dot{x} \right \rb - \left \lb \pder{\mathcal{I}}{x}(\dot{x}) , \delta x \right \rb }_{ T_1} + \underbrace{ \mathcal{I} \left( \frac{D\delta x}{D t} \right) - \mathcal{I} \left( \left. \frac{D \dot{x}}{ D\lambda} \right|_{\lambda = 0} \right) }_{T_2} \\
		& - [ \mathcal{I}( \delta x) , u ]_{\rm Jacobi-Lie}
	\end{align*}
	By \eqref{eq:da} $T_1 = d \mathcal{I}( \dot{x} , \delta x)$.  Moreover, by the torsion free property of the Levi-Cevita connection we find $\left. \frac{D \dot{x}}{ D \lambda } \right|_{\lambda = 0} = \frac{D \delta x}{Dt}$.  This yields $T_2 = 0$.  Finally substituting $u = \xi_x + \mathcal{I}(\dot{x})$ we can express $\delta \xi_x$ by
	\[
		\delta \xi_x = d \mathcal{I}( \dot{x} , \delta x) - [\mathcal{I}(\delta x) , \mathcal{I}(\dot{x}) ]_{\rm Jacobi-Lie} - [ \mathcal{I}( \delta x) , \xi_x ]_{\rm Jacobi-Lie}
	\]
	Now we apply Proposition \ref{prop:cov_der}.
\end{proof}

The reduced curvature tensor $\widetilde{B}$ measures the 
non-integrability of the distribution induced by $\mathcal{I}$.

Hamilton's principle considers variations of a curve 
$\varphi_t \in \SDiff(M)$. These variations of $\varphi_t$ 
induce a restricted class of variations of $(x,\xi_x)$.
As a result of Propositions \ref{prop:ver_var} and 
\ref{prop:hor_var}, we see that the most general covariant 
variations of $(x, \xi_x) \in \tilde{\mathfrak{g}}_\odot$ that 
would appear in a reduction of Hamilton's principle, are of 
the form
\begin{equation}
\label{good_variations}
\delta^\inter(x,\xi_x) = \frac{D(x,\eta_x)}{D t} +
[ (x,\eta_x),(x,\xi_x) ] + \widetilde{B}(\dot{x},\delta x)
\end{equation}
for some curve $(x,\eta_x) \in \tilde{\mathfrak{g}}_\odot$ and 
a variation $\delta x(t)$ of the curve $x(t)$.  In the next 
section we will state this reduced variational principle 
explicitly.

\subsection{Lagrange-Poincar\'{e} Reduction}
In this section, we state the Lagrange-Poincar\'{e} reduction
theorem (see~\cite{CeMaRa2001}) in terms of interpolation
methods rather than principal connections. The
resulting equations of motion are related to the 
Euler-Lagrange equations through the isomorphism, $\Psi_{\mathcal{I}} : (T \SDiff(M) ) / G_{\odot} \to TX \oplus \tilde{\mathfrak{g}}_\odot$, of 
Proposition \ref{prop:isomorphism}.
Note that if a Lagrangian $L:T\SDiff(M) \to \mathbb{R}$ is 
$G_\odot$-invariant, then there exists a smooth function 
$\ell: T\SDiff(M)/G_\odot \to \mathbb{R}$ defined by 
$\ell( [ \varphi, \dot{\varphi}] ) = 
L( \varphi, \dot{\varphi})$ for each 
$(\varphi, \dot{\varphi}) \in T\SDiff(M)$.  By using $\Psi_{\mathcal{I}}$ we can 
alternatively define $\ell$ as a function on $TX \oplus 
\tilde{\mathfrak{g}}_\odot$\footnote{We will abuse notation 
and use $\ell$ to simultaneously denote a function on 
$T\SDiff(M) / G_\odot$ and $TX \oplus 
\tilde{\mathfrak{g}}_\odot$.}. For example, the reduced 
Lagrangian of $L_{\Euler}$, denoted $\ell_{\Euler}: TX \oplus
\tilde{\mathfrak{g}}_\odot \to \mathbb{R}$, is given by
\begin{align}
\ell_{\Euler}(x,\dot{x},\xi_x) = \frac{\rho}{2} \int_{M}
{\|\xi_x + \inter(\dot{x})\|^2 d\mathrm{vol}(m) }. 
\label{eq:reduced_KE}
\end{align}
In order to write down the resulting equations of motion on 
the vector bundle $TX \oplus \tilde{\mathfrak{g}}_\odot$, it 
helps to note that there is a natural covariant derivative 
induced by the metric on $M$ and $\inter$.  In particular, we 
 introduce the Riemannian metric on $X$
given by
\[
\left\langle v_x, w_x \right\rangle_X :=
\int_M{ \left\langle \inter(v_x)(m) ,\inter(w_x)(m) 
\right\rangle_M d\mathrm{vol}(m) }.
\]
We may use the Levi-Civita connection associated to 
$\langle \cdot , \cdot \rangle_X$ to get a covariant 
derivative on $TX$.  Taking the direct sum of the covariant 
derivative on $X$ and the covariant derivative on 
$\tilde{\mathfrak{g}}_\odot$ produces a covariant derivative 
on $TX \oplus \tilde{\mathfrak{g}}_\odot$.  Let $\frac{D}{Dt}$
be the covariant derivative along a curve relative to
this direct sum covariant derivative on $M$ and
$\tilde{\mathfrak{g}}_\odot$.  
This is the final tool we require in order to express the 
reduced equations of motion.

\begin{thm} \label{thm:main1}
Let $L:T\SDiff(M) \to \mathbb{R}$ be a $G_\odot$-invariant 
Lagrangian {\rm (}for example, the kinetic energy Lagrangian
given in \eqref{eq:KE}{\rm )}.  Let $\inter:TX \to
\mathfrak{X}_{\mathrm{div}}(M)$ be
an interpolation method and  $\ell:TX \oplus
\tilde{\mathfrak{g}}_\odot \to\mathbb{R}$ the reduced
Lagrangian.  Let $\varphi_t
\in \SDiff(M)$ be a curve and set
$(x, \dot{x}, \xi_x)(t): = 
\Psi_\mathcal{I}\left(\left[\varphi_t, \dot{\varphi}_t
\right]\right)$. Then the following are equivalent.

\begin{itemize}
\item[{\rm (i)}] The curve $\varphi_t$ is critical for the action
\[
S = \int_0^1{L(\varphi_t, \dot{\varphi}_t) dt}
\]
with respect to variations $\delta \varphi_t$ with fixed end points.
\item[{\rm (ii)}] The curve $\varphi_t$ satisfies the Euler-Lagrange
equations
\[
\frac{D}{Dt} \left( \pder{L}{\dot{\varphi}_t} \right) -
\pder{L}{\varphi_t} = 0.
\]
with respect to an arbitrary covariant derivative on 
$T\SDiff(M)$.

\item[{\rm (iii)}] The curve $(x,\dot{x},\xi_x)(t)$ is critical for
the reduced action
\[
[S] = \int_0^1{\ell(x,\dot{x},\xi_x)(t) dt}
\]
with respect to arbitrary variations $\delta x(t)$, with
fixed end points, and covariant variations of $(x,\xi_x)$
of the form
\[
\delta^\inter (x,\xi_x) = \frac{D(x,\eta_x)}{Dt} +
[(x,\eta_x),(x,\xi_x)] + \widetilde{B}( \dot{x}, \delta x)
\]
for arbitrary curves $(x,\eta_x)(t) \in \tilde{\mathfrak{g}}_\odot$ which cover $x(t)$.
\item[{\rm (iv)}]  The curve $(x,\dot{x},\xi_x)(t)$ satisfies the
Lagrange-Poincar\'{e} equations
\begin{align*}
\frac{D}{Dt} \left( \pder{\ell}{\dot{x} } \right) -
\pder{\ell}{x} &= i_{\dot{x}} \widetilde{B}_{{\partial \ell}/ \partial \xi_x} \quad \text{on} \quad TX \\
\frac{D}{Dt} \left( \pder{\ell}{\xi_x} \right) &=
-\ad_{(x,\xi_x)}^*\left( \pder{\ell}{\xi_x} \right) \quad \text{on} \quad \tilde{\mathfrak{g}}_\odot
\end{align*}
where $\widetilde{B}_{{\partial \ell}/ \partial \xi_x}$ is the real-valued 2-form on $X$ given by $\widetilde{B}_{{\partial \ell}/ \partial \xi_x}(v_x,w_x) = \left\langle \dfrac{\partial \ell}{\partial \xi_x} , \widetilde{B}(v_x,w_x) \right\rangle$.
\end{itemize}
\end{thm}

\begin{proof}
The equivalence of (i) and (ii) is an intrinsic formulation of the standard derivation of the Euler-Lagrange equations from
Hamilton's variational principle (see, e.g., \cite{FOM},
\cite{Arnold2000}, \cite{MandS}).
We have chosen to write the intrinsic formulation more out of necessity than interest, as we are working on a space with non-trivial coordinate charts.
In this case, the ``equivalence of mixed partials'' comes from our definition of the covariant derivative induced by a generalized connection.
Specifically, let $\varphi_{t,\lambda}$ be an embedding of a surface into $\SDiff(M)$ (i.e., a deformation of a curve).  Then we observe that
\begin{align}
\frac{D\dot{\varphi_t}}{D \lambda} = v_{\downarrow}\left( \mathrm{ver} \left(\left.
\frac{\partial^2 \varphi_{t,\lambda} }{ \partial t \partial
    \lambda }\right|_{\lambda = 0} \right) \right) =
\frac{D \delta \varphi_t}{Dt} \label{eq:mixed_partials}
\end{align}
where $\dot{\varphi_t} = \pder{\varphi_{t,0}}{t}$ and $\delta
\varphi_{t} =
\restr{\pder{\varphi_{t,\lambda}}{\lambda}}{\lambda=0}$.
Using this observation, one is able to prove the equivalence of (i) and (ii) using a standard integration by parts argument (see \cite[Proposition 3.8.3]{FOM} for a proof with coordinates or \cite[Proposition 2.5]{JaVa2012} for a proof without coordinates).
\smallskip

We now prove the equivalence of (i) and (iii). By construction, the two actions $S$ and $[S]$ coincide on
the indicated curves.
If the action $S$ is extremized along $(\varphi_t,\dot{\varphi}_t)$, then
$[S]$ must be extremized along $(x,\dot{x},\xi_x)(t) = 
\Psi_\inter( [ (\varphi_t,\dot{\varphi}_t ) ] )$ (see
Proposition \ref{prop:isomorphism}) with respect to  variations induced by an arbitrary variation 
$\delta\varphi_t$.  In particular, by Proposition \ref{prop:hor_var}, the contribution to the covariant variation $\delta^\mathcal{I} \xi_x$ induced by the variation $\delta x$ is $\widetilde{B}( \dot{x}, \delta x)$.  Secondly, by Proposition \ref{prop:ver_var}, the contribution to $\delta^{\mathcal{I}} (x,\xi_x)$ due to the vertical component $\eta_x  = \delta \varphi_t \circ \varphi_t^{-1} - \mathcal{I}( \delta x )$ is given by $\frac{D}{Dt}( x, \eta_x) - [ (x, \xi_x) , (x, \eta_x) ]$.  In summary, $\delta^{\mathcal{I}} (x,\xi_x)$ is given by \eqref{good_variations}.  Thus (i) 
implies (iii).

Conversely, given curves $(x, \eta_x)(t) \in 
\tilde{\mathfrak{g}}_\odot$ and $\delta x(t) \in TX$
covering $x(t) \in X$, define 
\[
\delta \varphi(t): = \left(\eta_x(t) + 
\mathcal{I}(\delta x(t)\right)
 \circ\varphi_t
\]
and note that this is an arbitrary variation along the
curve $\varphi_t$. Since $S$ and $[S]$ coincide along the
curves $(\varphi_t, \dot{\varphi}_t)$ and $\Psi_\mathcal{I}(
[\delta\varphi_t, \dot{\varphi}_t]) = (x(t), \dot{x}(t), 
\xi_x(t))$, respectively, and 
\[
\Psi_\mathcal{I} \left([(\varphi, \delta\varphi)]\right) =
\left(x, \delta x,  \frac{D}{D t} (x, \eta_x) + [(x, \eta_x),
(x, \xi_x)] + \tilde{B}(\dot{x}, \delta x)\right), 
\]
we see that $\delta S = \delta[S]$. This shows that (iii) 
implies (i).
\smallskip

Finally, we prove the equivalence of (iii) and (iv).
Assume $[S]$ is extremized with respect to the variations given in equation
\eqref{good_variations}.
Then we find that
\[
    \delta [S] = \restr{ \pder{}{\lambda} }{\lambda = 0}
    \int_{0}^{1}{ \ell( x_\lambda, \dot{x}_\lambda,
      \dot{\xi}_{x,\lambda}) dt} 
    = \int_0^1{ \left\langle d\ell , \delta(x,\dot{x},\xi_x) \right\rangle dt}.
\]
We may use \eqref{eq:df=dmf+def} to rewrite the exterior derivative of $\ell$ using the covariant derivative on the vector bundle $TX \oplus \tilde{\mathfrak{g}}_\odot$ to arrive at
\[
\delta [S]
= \int_{0}^{1} \left({\left\langle \pder{\ell}{\xi_x} ,
    \left. \frac{D(x, \xi_x) }{D\lambda} \right|_{\lambda=0} \right\rangle +
      \left\langle\pder{\ell}{\dot{x}}, 
      \left.\frac{D \dot{x}}{D\lambda}\right|_{\lambda=0}
      \right\rangle + \left \langle \pder{
        \ell }{x} , \delta x \right\rangle}\right)dt.
\]
We can now make two substitutions.  Firstly, by definition $\delta^\mathcal{I} (x,\xi_x) := \left. \frac{D(x,\xi_x)}{D\lambda} \right|_{\lambda = 0}$, which we assume is given by \eqref{good_variations} for some curve $(x,\eta_x)_t \in \tilde{\mathfrak{g}}_\odot$ covering $x(t)$.  Secondly, if we invoke Definition \ref{def:generalized_connection} with respect to the Levi-Civita connection on $X$, we observe
\begin{align}
\left.\frac{D\dot{x}}{D\lambda}\right|_{\lambda=0}
\stackrel{\eqref{covariant_derivative}} = 
v_\downarrow \left( \ver \left( \left. \pder{}{\lambda} \right|_{\lambda=0} \left( \frac{\partial x}{\partial t} \right) \right) \right) 
	= v_\downarrow \left( \ver \left( \pder{}{t} \left(  \left. \pder{x}{\lambda} \right|_{\lambda = 0} \right) \right) \right) 
	= \frac{D \delta x}{Dt}\,. 
	\label{eq:mixed_partials}
\end{align}
If we substitute \eqref{good_variations} and \eqref{eq:mixed_partials} into our expression for $\delta [S]$ we find
\begin{align*}
\delta [S] &= \int_{0}^{1}{ \left\langle \pder{ \ell}{\xi_x} ,
\frac{D(x,\eta_x)}{Dt} + [(x,\eta_x),(x,\xi_x)] +
\widetilde{B}(\dot{x},\delta x) \right\rangle dt} \\
& \quad \quad +\int_{0}^{1}\left({ \left\langle \pder{ \ell }{\dot{x} }, \frac{ D \delta x}{Dt} \right\rangle +
\left\langle \pder{ \ell }{x} , \delta x \right\rangle}
\right)dt
\end{align*}
We separate $\delta [S]$ into a part which is proportional to $\delta x$ and a part which is proportional to $\eta_x$.  We will call these components $T_x$ and $T_\eta$ respectively, so that $\delta [S] = T_x + T_\eta$.  In particular, $\delta [S]$ is zero with respect to arbitrary variations $\delta x$ and $\eta_x$ if and only if $T_x$ and $T_\eta$ are identically $0$. We find that
\[
	T_\eta = \int_{0}^{1} \left \langle \pder{ \ell}{\xi_x} , \frac{D(x,\eta_x)}{Dt} + [ (x,\eta_x) , (x, \xi_x) ] \right \rangle dt.
\]
The definition of the covariant derivative on the dual-adjoint bundle $\tilde{ \mathfrak{g}}_\odot^*$ yields the equation
\[
	\left \langle \pder{ \ell}{\xi_x} , \frac{D (x,\eta_x)}{Dt} \right \rangle = \frac{d}{dt} \left \langle \pder{ \ell}{\xi_x} , (x,\eta_x) \right \rangle - \left \langle \frac{D}{Dt} \left( \pder{ \ell}{\xi_x} \right) , (x, \eta_x) \right \rangle
\]
which we  substitute into the expression for $T_\eta$ to find
\begin{align*}
T_\eta &=  \left. \left[ \left \langle \pder{ \ell}{\xi_x} , (x,\eta_x) \right \rangle \right] \right|_{t=0}^{t=1} \\
 & \quad - \int_0^1\left(\left\langle \frac{D}{Dt} \left( \pder{ \ell}{\xi_x} \right) , (x, \eta_x) \right \rangle - \left\langle \pder{ \ell}{ \xi_x} , [ (x, \eta_x) , (x, \xi_x) ] 
\right\rangle\right) dt
\end{align*}
As $\eta_x = 0$ at time $t=0,1$ the boundary term is $0$.  Moreover,
\begin{align*}
	\left \langle \pder{ \ell}{ \xi_x} , [ (x,\eta_x) , ( x, \xi_x) ] \right \rangle &= \left \langle \pder{\ell}{\xi_x} , - \ad_{(x,\xi_x)}( x,\eta_x) \right \rangle \\
		&= \left \langle - \ad^*_{(x,\xi_x) } \left( \pder{ \ell}{\xi_x} \right) , (x, \eta_x) 
		\right \rangle.
\end{align*}
We  substitute the above expression into $T_\eta$ and get
\[
	T_\eta =  - \int_{0}^{1} \left \langle \frac{D}{Dt} \left( \pder{ \ell}{\xi_x} \right) + \ad_{(x,\xi_x)}^* \left( \pder{ \ell}{\xi_x} \right) , (x,\eta_x) \right \rangle dt. 
\]
If $T_\eta = 0$ for an arbitrary $(x,\eta_x)$ then we find 
that the vertical equation
\[
  \frac{D}{Dt} \left( \pder{\ell}{\xi_x} \right) = - \ad^*_{(x,\xi_x) } \left( \pder{ \ell}{\xi_x} \right)
\]
must hold.

We now consider the equation $T_x = 0$.  We find that
\[
	T_x = \int_0^1\left( \left\langle \pder{ \ell}{\dot{x}} , \frac{D \delta x }{Dt} \right \rangle  + \left \langle \pder{\ell}{x} , \delta x \right \rangle + \left \langle \pder{\ell}{\xi_x} , \widetilde{B}( \dot{x} , \delta x) \right \rangle \right)dt.
\]
Using the definition of the Levi-Civita derivative on $T^*X$ we find
\[
	\left \langle \pder{\ell}{\dot{x}} , \frac{D \delta x}{Dt} \right \rangle = \frac{d}{dt} \left \langle \pder{\ell}{\dot{x}} , \delta x \right \rangle - \left \langle \frac{D}{Dt} \left( \pder{\ell}{\dot{x}} \right) , \delta x \right \rangle.
\]
We substitute this into our expression for $T_x$ to find
\[
	T_x = \left. \left[ \left \langle \pder{ \ell}{\dot{x}} , \delta x \right \rangle \right] \right|_{t=0}^{t=1} - 
\int_{0}^{1} \left(\left\langle \frac{D}{Dt} \left( \pder{ \ell}{\dot{x}} \right) - \pder{\ell}{x} , 
\delta x \right\rangle - \left\langle \pder{\ell}{\xi_x}, \widetilde{B}(\dot{x}, \delta x) \right\rangle \right) dt.
\]
As $\delta x$ vanishes at the end-points, we may ignore the boundary terms.  Moreover, the final term may be re-written as
\[
\left \langle \pder{ \ell}{\xi_x} , 
\widetilde{B}( \dot{x} , \delta x) \right \rangle =  
\left \langle i_{\dot{x}} \widetilde{B}_{{\partial \ell}/ \partial \xi_x} , 
\delta x \right \rangle.
\]
Thus we find
\[
	T_x = - \int_0^1 \left \langle \frac{D}{Dt} \left( \pder{ \ell}{\dot{x}} \right) - \pder{\ell}{x} - i_{\dot{x}} \widetilde{B}_{{\partial \ell}/ \partial \xi_x} , \delta x \right \rangle dt.
\]
If $T_x = 0$ for arbitrary fixed end-point variations, $\delta x$, then the horizontal equation
\[
	\frac{D}{Dt} \left( \pder{ \ell}{\dot{x}} \right) - \pder{\ell}{x} = i_{\dot{x}} \widetilde{B}_{{\partial \ell}/ \partial \xi_x}
\]
must hold.  Thus we have shown that (iii) implies (iv).  Additionally,
the above sequence of calculations is reversible and (iv) can be shown to imply (iii).
\end{proof}

At this point, we may be inspired to come up with particle methods by trying to better understand the horizontal equation (the equation for the dynamics on $TX$).
However, for the case of an incompressible homogenous ideal fluid one has the right to be very skeptical of this idea because the vertical equation expresses the overwhelming majority of the dynamics.
In fact, if we unpack the terms of the vertical equation for the Lagrangian $L_{\Euler}$, we find that $\pder{\ell_{\Euler}}{\xi_x} = \left\langle \xi_x + \inter(\dot{x}), \cdot \right\rangle_{L^2} = u^\flat$, so that the left hand side is
\begin{align*}
\frac{Du^\flat}{Dt} = \pder{u^\flat}{t} + \ad^*_{\inter(\dot{x})}(u^\flat),
\end{align*}
while the right hand side of the vertical equation is $-\ad^*_{(x,\xi_x)}(u^\flat)$.
Bringing both terms to one side we find $\pder{u^\flat}{t} + \ad_u^*u^\flat = 0$, which is the Euler equation for an
ideal homogeneous fluid equation (see, e.g., \cite{FOM} \cite{ArKh1992}), except for the fact that the vertical equations (strictly speaking) only address the domain complementary to the particle locations.
In essence, the horizontal equations for $L_{\Euler}$ only state that the particles move in such a way that $u$ can be extended smoothly by ``filling the holes''.
Despite this sobering observation, we know that computational scientists simulate fluids and successfully use interpolation methods frequently.
After studying what happens when we reduce by a class of subgroups of $G_\odot$ in \S \ref{sec:higher_order}, we will try to understand how the horizontal equation can potentially inspire particle methods in \S \ref{sec:particle_methods}.

\section{Higher Order Isotropy Groups} \label{sec:higher_order}
In the previous section we reduced our system by the Lie group
$G_\odot$.  This resulting in a set of coupled equations on the vector bundles $TX$
and $\tilde{\mathfrak{g}}_\odot$.  Notably, $X$ is the configuration manifold for the dynamics of
particles in $M$.  That data associated to an $x \in X$ consists of the values of $N$ points of some
diffeomorphism $\varphi \in \SDiff(M)$ which represents the configuration of the fluid.
We would like to take this further and construct particles which carry the values
of diffeomorphisms as well as a finite amount of derivative data.  In other words
we would like to have jet-data attached out particles.  This will be useful in the final section 
where we construct some plausible particle methods.

To obtain these data augmented particles we will consider the Lie group
\[
G_{\odot}^{(k)} := \{ \psi \in G_{\odot} \mid T^{(k)}_\odot
\psi \text{ is the identity on } T^{(k)}_\odot M \}.
\]
In local coordinates, elements of $G_\odot^{(k)}$ are
diffeomorphisms such that the Taylor expansion around each
point $\odot_i$ of $\odot$ is of the form
\[
	\psi( \odot_i + \delta x_i) = \odot_i + \delta x_i + o( \| \delta x_i \|^{k})
\]
  To make this more precise we
will include a short discussion on jet bundles---in particular, jets of elements in $\SDiff(M)$.


\subsection{Jet bundles of the special diffeomorphism group}
Consider the following equivalence relation:
$\varphi_1, \varphi_2 \in \operatorname{SDiff}(M)$ are
\emph{ equivalent to $k^\text{th}$ order at $x = 
(x_1, \ldots, x_N) \in X$} if,
in a chart (and hence all charts), they have the same Taylor
expansion at each $x_i$, $i=1, \ldots, N$.  Denote the
set of equivalence classes $j^k_x(\varphi)$ for this relation 
by $\mathcal{J}_x^k(\SDiff(M))$;  elements $j^k_x(\varphi)\in 
\mathcal{J}_x^k(\SDiff(M))$ are called $k$-jets of
$\SDiff(M)$ sourced at $x \in X$. For $k=0$, the equivalence
relation above states that $\varphi_1(x) = \varphi_2(x)$ and 
for $k = 1$ it  means that $T_x \varphi_1 = T_x \varphi_2$.  
For any $l < k$, there is a natural projection $\pi^k_l: 
\mathcal{J}^k_x(\SDiff(M)) \to
\mathcal{J}^l_x(\SDiff(M))$.  In particular,
$\pi^k_0 \left(j^k_x\varphi\right) =\varphi(x) \in X$. 
If $j^k_x(\varphi_1) = j^k_x(\varphi_2)$ and 
$j^k_y(\psi_1) = j^k_y(\psi_2)$ for $\varphi_1, \varphi_2, 
\psi_1, \psi_2 \in \SDiff(M)$, $y=\varphi_1(x)=\varphi_2(x)$,
then $j^k_x(\psi_1 \circ \varphi_1) = 
j^k_x(\psi_2 \circ \varphi_2)$. This leads to the definition
of jet composition $j^k_y(\psi)\circ j^k_x(\varphi) := j^k_x(\psi\circ \varphi)$. For more information on jet
spaces see, e.g., \cite[\S12]{KMS99}.
\medskip

There is an alternative way to define $k$-jets using the 
$k$th order tangent bundle, which we now define.
Note that a curve in $M$ is merely an
element of $C^{\infty}(I ; M)$ for some interval $I \subset \mathbb{R}$
which contains $0 \in \mathbb{R}$.   
Define the equivalence relation between curves on $M$ by 
$c_1( \cdot ) \sim_0 c_2 ( \cdot)$ if and only if 
$c_1( 0 ) = c_2(0)$.  Thus, the quotient 
$\frac{C^{\infty}(I;M)}{ \sim_0}$ is identified with $M$ itself.  
Next, consider the equivalence relation 
$c_1( \cdot ) \sim_1 c_2( \cdot)$ given by the conditions 
$c_1(0) = c_2(0)$ and $\left. \frac{dc_1}{dt} \right|_{t=0} = 
\left. \frac{dc_2}{dt} \right|_{t=0}$.  We observe that 
$TM = \frac{C^{\infty}(I;M)}{ \sim_1}$ (see, e.g., 
\cite[Definition 1.6.3]{FOM}, 
\cite[Definition 0.2.3]{Bleecker1981}, 
\cite[Definition 3.3.1]{MandS}).  Finally, define the 
equivalence relation $c_1( \cdot ) \sim_k c_2( \cdot)$ given 
by the conditions
\[
c_1(0) = c_2(0), \quad \left. \frac{dc_1}{dt} \right|_{t=0} = 
\left. \frac{dc_2}{dt} \right|_{t=0}, \quad \ldots, 
\quad \left. \frac{d^k c_1}{dt^k} \right|_{t=0} = 
\left. \frac{d^k c_2}{dt^k} \right|_{t=0}
\]
in a chart (and hence all compatible charts at $c_1(0)$).
The \emph{$k$th order tangent bundle} is defined as the 
quotient space $T^{(k)}M := \frac{ C^{\infty}(I;M) }{ \sim_k }$ 
equipped with the fiber bundle projection $\tau^{(k)}: T^{(k)} 
M \to M$.  Given a curve $c \in C^{\infty}(I;M)$ denote its 
equivalence class with respect to $\sim_k$ by $[c]_k$; thus 
the fiber projection is given by $\tau^{(k)}( [c]_k) = c(0)$.
Note that, with the exception of $k=0,1$,  $\tau^{(k)}: 
T^{(k)} M \to M$ are not vector bundles. 

Note that any $\varphi \in \SDiff(M)$ acts on a curve 
$c \in C^{\infty}(I;M)$ by composition.  Define the diffeomorphism 
$T^{(k)} \varphi : T^{(k)}M \to T^{(k)}M$ by 
$T^{(k)} \varphi([c ]_k) := [\varphi \circ c]_k$.  Given 
$\varphi \in \SDiff(M)$ and a point $m \in M$ we see that 
$T_m^{(k)}\varphi:(\tau^{(k)})^{-1}(m)\rightarrow 
(\tau^{(k)})^{-1}(\varphi(m))$. Note that
$T^{(k)} (\varphi_2 \circ \varphi_1) = T^{(k)} \varphi_2  
\circ T^{(k)} \varphi_1$.  If we pay attention to the base 
points, this last equation reads $T^{(k)}_m (\varphi_2 \circ 
\varphi_1) = T^{(k)}_{\varphi_1(m)} \varphi_2 \circ T^{(k)}_m 
\varphi_1$. Following our conventions in the previous section, 
we define $T_{\odot}^{(k)} \varphi : =  
\left(T_{\odot_1}^{(k)} \varphi, 
\ldots, T_{\odot_N}^{(k)} \varphi \right)$.

Define the subgroup
\begin{align*}
G^{(k)}_\odot :&= \{ \varphi \in \SDiff(M)\mid 	 
T^{(k)}_{\odot} \varphi = T^{(k)}_\odot id  \}\\
& = \left\{\psi \in G_\odot \;\left|  \; 
\frac{\partial \psi^i}{\partial m^j}(\odot_l) = \delta_i^j,\;
\;\frac{\partial^{|\alpha|} \psi^i}{\partial m^\alpha} 
(\odot_l)= 0, \right. 1<|\alpha| \leq k, \; l=1, \ldots, N
\right\},
\end{align*}
where $T^{(k)}_{\odot_j} id$ is the identity map on the fiber 
$T^{(k)}_{\odot_j}M$, $\alpha:= (\alpha_1, \ldots, 
\alpha_{\dim(M)})$ is a multi-index, $\alpha_j \geq 0$, 
$\alpha\in\mathbb{N}^{\dim(M)}$,  $|\alpha|: = 
\alpha_1 + \cdots + \alpha_{\dim(M)}$, and 
$\frac{\partial^{|\alpha|}}{\partial m^\alpha} : = 
\frac{\partial^{|\alpha|}}{(\partial m^1)^{\alpha_1} 
(\partial m^2)^{\alpha_2} \cdots 
(\partial m^{\dim(M)})^{\alpha_{\dim(M)}}}$. The second
equality is a direct verification in a local chart.
The subgroup $G^{(k)}_\odot$ is itself a Lie group, and we denote its Lie algebra by $\mathfrak{g}^{(k)}_\odot$, namely
\[
\mathfrak{g}_\odot^{(k)} := \{ \xi \in \mathfrak{g}_\odot \mid \left. \partial_\alpha \right|_{\odot} \xi = 0,\;   0 \leq | \alpha | \leq k \},
\]
i.e., $\mathfrak{g}^{(k)}_\odot$ consist of vector fields 
$\xi \in \mathfrak{X}_{\mathrm{div}}(M)$ which vanish at the points $\odot$ to $k$th order.

In the following sections we will consider reducing Lagrangian 
systems on $\SDiff(M)$ by this subgroup.  In order to relate 
these ideas to the reductions performed earlier in the paper, 
we note the following property.
  
\begin{prop} \label{prop:normal_subgroup}
The subgroup $G^{(k)}_\odot$ is normal in $G_\odot$.
\end{prop}
\begin{proof}
Let $\varphi \in G_\odot$ and $\psi \in G^{(k)}_\odot$. Thus $T_\odot^{(k)} \psi = T^{(k)}_\odot id$
and we find
\[
T^{(k)}_\odot (\varphi \circ \psi \circ \varphi^{-1}) = 
T^{(k)}_\odot \varphi \circ T^{(k)}_\odot \psi \circ 
(T^{(k)}_\odot \varphi)^{-1} = T^{(k)}_{\odot} id.
\]
 This shows $\varphi \circ \psi \circ \varphi^{-1} \in G_\odot^{(k)}$.
\end{proof}

The subgroup $G^{(k)}_\odot$ acts on $\SDiff(M)$ by 
composition on the right. Note that by the relevant definitions, $j^k_\odot \varphi = j^k_\odot(id)$ if and 
only if $T^{(k)}_\odot \varphi = T^{(k)}_\odot id$. This immediately
shows that both maps
\[
\mathcal{J}_\odot^{k}( \SDiff(M)) \ni j^k_\odot \varphi \longleftrightarrow
[\varphi] \in \SDiff(M) / G^{(k)}_\odot.
\]
are well defined and inverses to each other. Thus, we can
identify $\mathcal{J}_\odot^{k}( \SDiff(M))$ with 
$\SDiff(M) / G^{(k)}_\odot$ which we shall do in the rest
of the paper. Using this identification, Proposition
\ref{prop:normal_subgroup} implies that 
\[
\mathcal{J}_\odot^{k}( G_\odot ) = G_\odot / G^{(k)}_\odot
\]
is a finite dimensional Lie group with Lie algebra $\mathfrak{g}_\odot / \mathfrak{g}_\odot^{(k)}$.
In particular, $\mathcal{J}^{k}_\odot( G_\odot)$ consists of $k$th order 
Taylor series data. We will use the right action of 
$\mathcal{J}_\odot^{k}(G_\odot)$ on 
$\mathcal{J}^{k}(\SDiff(M))$ in 
\S \ref{sec:particle_methods} to find conserved momenta 
for a class of particle methods, using Noether's theorem.

Finally, we would like to relate the tangent bundle of $\SDiff(M)$ to the tangent bundle of $\mathcal{J}_\odot^{k}(\SDiff(M))$.  Given any $\dot{\varphi} \in T\SDiff(M)$ we may consider the $k$-jet $j^k_{\odot}(\dot{\varphi}) \in T \mathcal{J}^k_{\odot}( \SDiff(M))$ as the vector
\[
	j^k_{\odot}(\dot{\varphi} ) = \left. \frac{d}{dt}\right|_{t=0} j^k_{\odot}( \varphi_t)
\]
for an arbitrary curve $\varphi_t \in \SDiff(M)$ such that $\left. \frac{d \varphi_t}{dt} \right|_{t = 0}= \dot{\varphi}$.
Of course, this definition gives a meaning to the expression $j^k_\odot(u)$ when $u \in \mathfrak{X}_{\mathrm{div}}(M) \subset T\SDiff(M)$ as well. 
In particular, we conclude that $\mathfrak{g}^{(k)}_\odot = \{ u \in \mathfrak{X}_{\mathrm{div}}(M) \mid j^k_\odot(u) = 0 \}.$

Given any $\varphi \in \SDiff(M)$ such that 
$\varphi(\odot) = x^{(0)} \in X$ we see that 
$j^{k}_\odot( u \circ \varphi)$ is an element
of $T\mathcal{J}_\odot^k( \SDiff(M))$.  Moreover, we may consider the
entire equivalence class, $x^{(k)}: = j^k_\odot \varphi$, as a set of maps
which can be composed with the set of maps comprising 
$j^k_{x^{(0)}} (u)$  to define
$j^{k}_{x^{(0)}}(u)\circ x^{(k)}:=j^k_\odot (u\circ \varphi)$.

Define the vector bundle $\tilde{\mathfrak{g}}_\odot^{(k)}$ with base $X^{(k)}:= 
\SDiff(M)/G_\odot^{(k)} = \mathcal{J}^{(k)}(\SDiff(M))$ by
\[
\tilde{\mathfrak{g}}^{(k)}_\odot := \left. 
\left\{ \left(x^{(k)} , \xi^{(k)}_{x^{(k)}}\right) \;\right|\; 
x^{(k)} \in X^{(k)} ,\; x^{(0)}: = 
\pi^k_0 \left(x^{(k)}\right) \in X,\;
\xi^{(k)}_{x^{(k)}} \in \mathfrak{g}^{(k)}_{x^{(0)}} \right\},
\]
where $\mathfrak{g}^{(k)}_{x^{(0)}}$ is the Lie algebra of 
the isotropy group $G^{(k)}_{x^{(0)}}$ for each 
$x^{(0)} \in X$, i.e., the set of divergence
free vector fields which vanish up to order $k$ at the $N$ 
particle locations $x^{(0)} \in X \subset M^N$
\footnote{The bundle $\tilde{\mathfrak{g}}^{(k)}_\odot$ is identical to the adjoint bundle $\frac{ \SDiff(M) \times 
\mathfrak{g}_\odot^{(k)} }{ G_\odot^{(k)} }$ when equipped with the fiberwise Lie bracket given by the negative of the
Jacobi-Lie bracket of vector fields.}.

\subsection{Quotients and $k^\text{\rm th}$ order interpolation methods}
In this section we generalize the notions of an interpolation method to handle jet-data to deal with reduction by $G^{(k)}_\odot$.
Define the quotient manifold $X^{(k)} := \SDiff(M) / G^{(k)}_\odot \equiv \mathcal{J}^{(k)}_\odot( \SDiff(M))$, which consists of particles augmented by Taylor series data.
For the case $k=1$, an element of $X^{(1)}$ is given by a 
$k$-tuple of unit-volume frames above non-overlapping 
points in $M$.   Therefore, $X^{(1)}$ is equivalent to the configuration manifold of particles with orientations and shape.  We can view orientation and shape as $1^{\rm st}$ order deformations of `infinitesimal balls'.  If $M = \mathbb{R}^d$ we identify the unit volume frames 
with $\SL(d, \mathbb{R})$ and $X^{(1)}=X^{(0)} \times 
\SL(d, \mathbb{R})^N
\rightarrow X^{(0)}$, a trivial fiber-bundle.
For $k > 1$ we obtain quantities which express higher-order deformations.   
%

Next, we follow the procedure
of Lagrange-Poincar\'{e} reduction used in the last section
but for a $k^{\rm th}$ order interpolation method, formally
defined below.

\begin{definition} \label{def:kth_order_inter}
A \emph{$k^\text{th}$ order interpolation method} is a $\mathfrak{X}_{\rm div}(M)$ valued one-form,
$\inter \in \Omega^1(X^{(k)}; \mathfrak{X}_{\mathrm{div}}(M) )$, such
that $j^k_{ x^{(0)}} \left(\inter( \dot{x}^{(k)}) \right) \circ  x^{(k)} = \left(x^{(k)},\dot{x}^{(k)}\right)$ where 
$\left(x^{(k)},\dot{x}^{(k)}\right) \in
TX^{(k)}$, and $x^{(0)} = \pi^k_0 \left(x^{(k)}\right)$.
\end{definition}

Note that $k=0$ recovers Definition 
\ref{def_interpolation}.  For $k > 0$ we get the natural higher order generalization of \ref{def_interpolation}.
In other words, a $k^\text{th}$ order interpolation method produces a vector field with given $k^\text{th}$ order Taylor-series data at a finite set of points.
Again, as in remark \ref{rmk:connections}, a $k^{\rm th}$ order interpolation method is equivalent to a right $G^{(k)}_\odot$-principal connection on $\SDiff(M)$.

By using a $k^\text{th}$ order interpolation method, we can follow the same constructions and procedures as in Sections \ref{sec:preliminary} and  \ref{sec:reduced_system}, beginning
with the definition of the fiber bundle isomorphism
$\Psi^{(k)}_\inter: T\SDiff(M) / G_{\odot}^{(k)} \to TX^{(k)} 
\oplus \tilde{\mathfrak{g}}_\odot^{(k)}$ (the analogue of 
the one in Proposition \ref{prop:isomorphism})  given by
\begin{equation}
\label{psi_k}
\Psi^{(k)}_\inter( [v_\varphi] ) = j^k_\odot( v_\varphi ) \oplus \left( v_\varphi
\circ \varphi^{-1} - \inter(j^k_\odot( v_\varphi ) ) \right).
\end{equation}
Working in this manner we may repeat all the proofs of the previous section by
using the jet operator $j^k_\odot$ in the right places. For example, the evaluation
 $x=\varphi(\odot)$ is now replaced by the evaluation 
$x^{(k)}=j^k_\odot \varphi$. Thus, all symbolic manipulations are entirely the same as in the previous sections.  For example, the reduced Lagrangian of $L_{\Euler}$ is now given by
\[
\ell_{\Euler}^{(k)}\left(x^{(k)}, \dot{x}^{(k)},  
\xi^{(k)}_{x^{(k)}}\right) = 
\frac{1}{2} \int_{M}{\left\|\mathcal{I}\left(\dot{x}^{(k)} \right) + \xi^{(k)}_{x^{(k)}}\right\|^2 d \mathrm{vol}(m) }.
\]

In particular, we obtain a modest generalization of Theorem \ref{thm:main1}.

\begin{thm} \label{thm:main2}
Let $L:T\SDiff(M) \to \mathbb{R}$ be a $G_\odot^{(k)}$-invariant Lagrangian {\rm (}e.g., $L_{\Euler}$
given in \eqref{eq:KE}{\rm )}.  Let $\mathcal{I} \in \Omega^1( X^{(k)} ;
\mathfrak{X}_{\mathrm{div}}(M) )$ be
a $k^\text{th}$-order interpolation method and 
$\ell^{(k)}:TX^{(k)} \oplus
\tilde{\mathfrak{g}}^{(k)}_\odot \to\mathbb{R}$ the reduced
Lagrangian defined by $\ell^{(k)} \circ \psi_\mathcal{I}: = L$.
Let $\varphi_t\in \SDiff(M)$ be a curve and set 
$\left(x^{(k)}, \dot{x}^{(k)}, \xi^{(k)}_{x^{(k)}}\right)(t) = \Phi^{(k)}_{\inter}( [ \varphi_t , \dot{\varphi}_t])$.
  Then the following are equivalent.
\begin{itemize}
\item[{\rm (i)}] The curve $\varphi_t$ is critical for the action
\[
S = \int_0^1{L(\varphi_t, \dot{\varphi}_t) dt}
\]
with respect to variations $\delta \varphi_t$ with fixed end points.
\item[{\rm (ii)}] The curve $\varphi_t$ satisfies the Euler-Lagrange
equations
\[
\frac{D}{Dt} \left( \pder{L}{\dot{\varphi}_t} \right) -
\pder{L}{\varphi_t} = 0.
\]
with respect to an arbitrary covariant derivative and connection on $T\SDiff(M)$.

\item[{\rm (iii)}] The curve $\left(x^{(k)}, \dot{x}^{(k)}, \xi^{(k)}_{x^{(k)}}\right)(t)$ is critical for
the reduced action
\[
[S] = \int_0^1{\ell^{(k)}\left(x^{(k)}, \dot{x}^{(k)}, \xi^{(k)}_{x^{(k)}}\right)(t) dt}
\]
with respect to arbitrary variations $\delta x^{(k)}(t)$, with
fixed end points, and covariant variations of 
$\left(x^{(k)}, \xi^{(k)}_{x^{(k)}}\right)(t)$
of the form
\[
\delta^\inter \left(x^{(k)}, \xi^{(k)}_{x^{(k)}}\right) = \frac{D\left(x^{(k)}, \eta^{(k)}_{x^{(k)}}\right)}{Dt} +
\left[\left(x^{(k)}, \eta^{(k)}_{x^{(k)}}\right),
\left(x^{(k)}, \xi^{(k)}_{x^{(k)}}\right)\right] + \widetilde{B}\left(\dot{x}^{(k)}, \delta x^{(k)}\right)
\]
for arbitrary curves $\left(x^{(k)}, \eta^{(k)}_{x^{(k)}}\right)(t) \in \tilde{\mathfrak{g}}_\odot^{(k)}$ covering $x(t)$.
\item[{\rm (iv)}]  The curve $\left(x^{(k)}, \dot{x}^{(k)}, \xi^{(k)}_{x^{(k)}}\right)(t)$ satisfies the
Lagrange-Poincar\'{e} equations
\begin{align}
\frac{D}{Dt} \left( \pder{\ell^{(k)}}{\dot{x}^{(k)}}\right) -
\pder{\ell^{(k)}}{x^{(k)}} &= 
\widetilde{B}_{{\partial \ell^{(k)}}/ \partial \xi^{(k)}_{x^{(k)}}}\left(\dot{x}^{(k)}, \cdot \right) 
\quad \text{on} \quad T_{x^{(k)}(t)}X^{(k)}  
\label{eq:hor_LP}\\
\frac{D}{Dt} \left(\pder{\ell^{(k)}}{\xi^{(k)}_{x^{(k)}}} 
\right) &=
-\ad_{\xi^{(k)}_{x^{(k)}}}^*\left(
\pder{\ell^{(k)}}{\xi^{(k)}_{x^{(k)}}}\right) \quad 
\text{on} \quad \tilde{\mathfrak{g}}_\odot^{(k)}(x^{(k)}(t)) \label{eq:ver_LP}
\end{align}
where $\widetilde{B}_{\partial \ell^{(k)}/ 
\partial \xi^{(k)}_{x^{(k)}}}
\left(v_{x^{(k)}},w_{x^{(k)}}\right) =  \left\langle 
\frac{\partial \ell^{(k)}}{\partial \xi^{(k)}_{x^{(k)}}} , 
\widetilde{B}\left(v_{x^{(k)}},w_{x^{(k)}}\right) 
\right\rangle$, for any $v_{x^{(k)}}, w_{x^{(k)}} \in 
T_{x^{(k)}}X^{(k)}$ and all $x^{(k)} \in X^{(k)}$.
\end{itemize}
\end{thm}

\begin{proof}
	The proof of theorem \ref{thm:main2} is nearly identical to the proof of Theorem \ref{thm:main1}.  The symbolic manipulations are the same upon substituting the expression `$x = \varphi(\odot)$' with the expression `$x = j^k_{\odot}( \varphi)$' into the proof of Theorem \ref{thm:main1}.
\end{proof}

When $k > 0$, the space $X^{(k)}$ stores particle positions, as well as $k^\text{th}$ order jet-data about the particle positions.
Moreover, the infinite dimensional vector bundle $\tilde{\mathfrak{g}}_\odot^{(k)}$ stores vector fields which vanish at the particle positions up to $k^\text{th}$ order.
Specifically, the \emph{horizontal equation} \eqref{eq:hor_LP} evolves the momenta of the particles and their jet-data, while the \emph{vertical equation} \eqref{eq:ver_LP} evolves the momenta conjugate to $(x, \xi^{(k)}_{x^{(k)}})$.

\begin{remark}{\rm 
 The vector field $\xi^{(k)}_{x^{(k)}}$ can be viewed as a residual if we interpret $\inter(\cdot)$ as an estimate of the fluid motion.  This interpretation of $\xi^{(k)}_{x^{(k)}}$ as a residual of an estimate will become more important when we discuss how $\inter$ induces particle methods for Lagrangian systems on $\SDiff(M)$.  In particular, the error analysis of these methods can be performed through this interpretation.}
 \end{remark}

\begin{remark} \label{rmk:higher_order}
{\rm
If $L$ is a $G_\odot$-invariant Lagrangian, then there is an extra symmetry in the horizontal equations \eqref{eq:hor_LP}.
In particular, the group $\mathcal{J}_\odot^k(G_\odot) \equiv G_\odot / G^{(k)}_\odot$ is a residual symmetry left over by a $G^{(k)}_\odot$ reduction.
For example, if $M = \mathbb{R}^d$ and $k=1$ we can identify $\mathcal{J}_\odot^1(G_\odot)$ with $\SL(d, \mathbb{R}) ^N$ by choosing unit volume frames above $\odot_1, \dots, \odot_N$.
Therefore, if the Lagrangian is independent of the orientation and shape of the particles, we get a right $\SL(d, \mathbb{R})^N$ symmetry which has yet to be ``quotiented away''.
This extra symmetry is attached to the particles, and results in conserved momenta as a consequence of Noether's theorem.
We will discuss this more in \S \ref{sec:kelvin} in the context of a particle method.}
\end{remark}

\section{Particle Methods} \label{sec:particle_methods}
We hope that a computational scientist or engineer interested
in particle methods for fluids finds the Lagrange-Poincar\'{e}
equations of Theorems \ref{thm:main1} and \ref{thm:main2} thought provoking;  we have deliberately presented
our approach using the ``interpolation'' point of view to
encourage such a reaction from practitioners.  In this section
we describe methods by which one can construct particle
methods.  This is not to say that these are the only methods
one can do!  Study along these lines is wide open for further
exploration and implementation.

In the previous section we considered the Lie subgroup
$G_\odot^{(k)} \subset \SDiff(M)$. The quotient $X^{(k)} :=\mathcal{J}_\odot^{k}( \SDiff(M)) = \SDiff(M) / G_\odot^{(k)}$ is a configuration manifold of particles which carry Taylor-series data.  The Lie algebra of $G^{(k)}_\odot$, which we denote $\mathfrak{g}_\odot^{(k)}$, consists of vector fields which vanish to $k^{\rm th}$ order at each of the points in $\odot$.  We then can define the associated bundle $\tilde{\pi}^{(k)}: \tilde{\mathfrak{g}}_\odot^{(k)} \to X^{(k)}$ by defining the fiber over $x^{(k)} \in X^{(k)}$ by
\[
	\tilde{\mathfrak{g}}_\odot^{(k)}\left(x^{(k)}\right) = \{ u \in \mathfrak{X}_{\mathrm{div}}(M) \mid j^k_{x^{(0)}}( u) = 0 \}, \quad x^{(0)} = \pi_0^{(k)}\left(x^{(k)}\right).
\]
In Theorem \ref{thm:main2}, we found equations of motion on the space $TX^{(k)} \oplus \tilde{\mathfrak{g}}_\odot$ (the Lagrange-Poincar\'{e} equations).  These equations of motion are given by a horizontal equation, \eqref{eq:hor_LP}, which describes the dynamics on $TX^{(k)}$, and a vertical equation, \eqref{eq:ver_LP}, which describes the dynamics of the $\tilde{\mathfrak{g}}_\odot^{(k)}$-component.
In particular, \eqref{eq:hor_LP} is similar to an 
Euler-Lagrange equation except for the coupling term involving the curvature on the right hand side and the 
$(x^{(k)},\xi^{(k)}_{x^{(k)}})$-dependence of particle momenta, $\partial \ell^{(k)}/\partial x^{(k)}$.
Therefore, it is tempting to find a way of ignoring the
$\xi^{(k)}_{x^{(k)}}$ and curvature terms in the hope of deriving a true Euler-Lagrange equation on the finite dimensional configuration manifold $X^{(k)}$.
One way to ignore the $\xi^{(k)}_{x^{(k)}}$-dependence is to non-holonomically constrain it to be $0$.  That means constraining the spatial velocity $u = \dot{\varphi}\circ \varphi^{-1}$ to the range of $\inter$.
Thus we introduce the constraint distribution $\Delta_{\inter} \subset T\SDiff(M)$ whose fiber over 
$\varphi \in \SDiff(M)$ is
\begin{align}
\Delta_{\inter}(\varphi) := 
\left.\left\{ \inter\left(\dot{x}^{(k)}\right) \circ \varphi \;\right|\; \dot{x}^{(k)} \in T_{x^{(k)}} X^{(k)}, x^{(k)} = j^k_\odot \varphi \right\}. \label{eq:constraint_distribution}
\end{align}
By construction, the non-holonomic force which keeps
$\dot{\varphi}$ in $\Delta_\mathcal{I}$ is such that the
vertical component $\xi^{(k)}_{x^{(k)}} = 
\dot{\varphi} \circ\varphi^{-1} -
\inter\left(\dot{x}^{(k)}\right)$ vanishes, where 
$\dot{x}^{(k)} = j^{(k)}_\odot(\dot{\varphi})$.  Thus, the dynamics
on $X^{(k)}$ determines the dynamics completely.  This becomes evident in the following theorem.  In particular, items (ii) and (iv) of Theorem \ref{thm:particle_method} can be seen as spatially $k^{\rm th}$ order accurate particle methods.

\begin{thm} \label{thm:particle_method}
Let $L:T\SDiff(M) \to \mathbb{R}$ be a 
$G_{\odot}^{(k)}$-invariant Lagrangian and 
$\mathcal{I} \in \Omega^1( X^{(k)} ; \mathfrak{X}_{\mathrm{div}}(M) )$ 
a $k^{\rm th}$-order interpolation method.  Let 
$\Delta_{\inter} \subset T\SDiff(M)$ be the constraint 
distribution given in \eqref{eq:constraint_distribution} 
and $F:TX^{(k)} \to T^{\ast}X^{(k)}$ the force defined by
\begin{align}
\left\langle F\left(x^{(k)},\dot{x}^{(k)}\right) , 
\delta x^{(k)} \right\rangle
= -\left. \frac{d}{d\epsilon} \right|_{\epsilon = 0 } 
\ell^{(k)}\left( x^{(k)},\dot{x}^{(k)}, 
 \epsilon \tilde{B}\left(\dot{x}^{(k)},\delta x^{(k)}\right)\right). \label{eq:force}
\end{align}
Finally, let $\varphi_t \in \SDiff(M)$ be a curve, set 
$\left(x^{(k)},\dot{x}^{(k)}, \xi^{(k)}_{x^{(k)}}\right)(t) = \Phi_{\inter}( [(\varphi_t, \dot{\varphi}_t)] ) \in TX^{(k)} \oplus \tilde{\mathfrak{g}}_\odot^{(k)}$, and define  the Lagrangian $L_{X^{(k)} }:TX^{(k)} \to \mathbb{R}$ by $L_{X^{(k)}}\left(x^{(k)},\dot{x}^{(k)} \right) = 
\ell^{(k)}\left(x^{(k)},\dot{x}^{(k)},0 \right)$.  Then the following are equivalent:
\begin{itemize}
\item[{\rm (i)}] $\varphi_t$ satisfies the constrained Euler-Lagrange equation
\[
\frac{D}{Dt} \left( \pder{L}{\dot{\varphi}_t } \right) - 
\pder{L}{\varphi_t} \in \Delta_\inter^\circ \quad , \quad 
\dot{\varphi}_t \in \Delta_{\inter}
\]
where $\Delta_{\inter}^{\circ} \subset T^{\ast}\SDiff(M)$ is 
the annihilator of $\Delta_{\inter}$. 
\item[{\rm (ii)}] $\xi^{(k)}_{x^{(k)}}(t) = 0$ and 
$x^{(k)}(t)$ satisfies the Lagrange-d'Alembert equations
\begin{align}
\frac{D}{Dt} \left( \pder{L_{X^{(k)}}}{\dot{x}^{(k)}}\right) - 
\pder{L_{X^{(k)}}}{x^{(k)}} = F.  
\label{eq:LDA}
\end{align}
\item[{\rm (iii)}] $\dot{\varphi}_t \in \Delta_{\inter}$ 
and $\varphi_t$ satisfy the variational principle
\[
\delta \int_0^1{ L(\varphi, \dot{\varphi} ) dt} = 0
\]	
with respect to variations $\delta \varphi_t \in 
\Delta_{\inter}$ which vanish at the endpoints.
\item[{\rm (iv)}] $\xi^{(k)}_{x^{(k)}}(t) = 0$ and 
$x^{(k)}(t)$ satisfies the Lagrange-d'Alembert principle
\[
\delta \int_0^1{L_{X^{(k)}}\left(x^{(k)},
\dot{x}^{(k)}\right)dt} 
= \int_0^1 \left\langle F\left(x^{(k)},\dot{x}^{(k)}\right) , 
\delta x^{(k)} \right\rangle dt
\]
with respect to arbitrary variations $\delta x^{(k)}$ with fixed endpoints.
  \end{itemize}
\end{thm}

\begin{proof}
The equivalence of (i) with (iii) and (ii) with (iv) is 
standard~\cite[Theorem 5.2.2]{bloch2003}. The equivalence 
of (iii) with (iv) is a result of the change of variables 
$\left(x^{(k)},\dot{x}^{(k)},\xi^{(k)}_{x^{(k)}}\right)(t) = 
\Psi^{(k)}_{\inter}\left(\left[(\varphi_t ,\dot{\varphi}_t)
\right]\right)$ (see \eqref{psi_k}).  In particular, assume (iii).  Then 
$(\varphi_t, \dot{\varphi}_t) \in \Delta_{\inter}$ and by  \eqref{eq:constraint_distribution} we observe
(suppressing the $t$-dependence),
\[
\xi^{(k)}_{x^{(k)}} 
\stackrel{\eqref{psi_k}}= 
\dot{\varphi} \circ \varphi^{-1} 
- \inter\left(\dot{x}^{(k)}\right) \stackrel{\eqref{eq:constraint_distribution}}{=} 0.
  \]
Moreover, if we consider a variation $\delta \varphi \in 
\Delta_{\inter}$, then the corresponding variation of 
$x^{(k)}$ is $\delta x^{(k)} = \delta j^k_\odot (\varphi)$.  The covariant variation of 
$(x^{(k)} , \xi^{(k)}_{x^{(k)}} )$ is 
$\delta^{\inter}(x^{(k)}, \xi^{(k)}_{x^{(k)}} ) = 
\tilde{B}\left(\dot{x}^{(k)} , \delta x^{(k)}\right)$ 
(by the obvious generalization of Proposition \ref{prop:hor_var} to $k^{\rm th}$ order jets).  By substitution and the defining relation $\ell^{(k)} \circ 
\Psi^{(k)}_\mathcal{I} = L$, we find
\[
0 = \delta \int_0^1 L(\varphi, \dot{\varphi} ) dt = 
\delta \int_0^1 \ell^{(k)}\left(x^{(k)},\dot{x}^{(k)},
\xi^{(k)}_{x^{(k)}}\right) dt .
\]
However,along a curve where 
$\xi^{(k)}_{x^{(k)}}(t) = 0$ the right hand side takes the form
\[
\delta \int_0^1 L_{X^{(k)}}(x,\dot{x}) dt  - 
\int_0^1\left\langle F\left(x^{(k)},\dot{x}^{(k)}\right) , \delta x^{(k)} \right\rangle dt
\]
by \eqref{eq:force}. Thus (iii) implies (iv).  

To prove the converse, note that the computations above
can be reversed by deriving an arbitrary variation 
$\delta \varphi \in \Delta_{\inter}$ from an arbitrary variation $\delta x^{(k)}$ via $\delta \varphi = 
\inter\left(\delta x^{(k)}\right) \circ \varphi$.  Thus (iii) and (iv) are equivalent.
\end{proof}

Theorem \ref{thm:particle_method}(ii) suggests a particle method obtained by solving a set of Lagrange-d'Alembert equations on the configuration manifold $X^{(k)}$ which can be lifted to the solution of a non-holonomically constrained Euler-Lagrange equation on $\SDiff(M)$.
The solution of constrained Euler-Lagrange equations differs from the solution of the unconstrained equations on $\SDiff(M)$ by an amount proportional to the magnitude of the constraint force required to keep $\dot{\varphi}$ inside $\Delta_{\inter}$.
This is discussed further in the following remark.
\begin{remark}{\rm 
If $M= \mathbb{R}^d$, the constraint force is given by the (co)vector field $\left(\inter(\dot{x}) \cdot \nabla\right)(\inter(\dot{x}))$.
Taking the norm of this quantity and using Grownwall's inequality, one can find an error bound and a stopping criterion for our particle method.} \quad $\lozenge$
\end{remark}

While Theorem \ref{thm:particle_method}(ii) can be used to construct a particle method for a fluid, the external force $F$ is potentially disturbing.
For example, certain guarantees such as conservation of energy, which hold for the Euler-Lagrange equations, do not hold for Lagrange-d'Alembert equations. 
In the case where a mechanical connection can be associated to the Lagrangian (see \S\ref{ex:interpolation_method}), this force, $F$, vanishes.
We will prove this in Corollary \ref{cor:particle_method}.
Unfortunately, there does not exist a mechanical connection in the case of an Euler fluid.
However, one could consider perturbing the Lagrangian $L_{\Euler}$ by a small parameter $\alpha > 0$ 
to obtain a new Lagrangian $L_\alpha$ which does exhibit a mechanical connection.

The existence of the mechanical connection is equivalent to the existence of an interpolation method $\mathcal{I} \in \Omega^1( X^{(k)} ; \mathfrak{X}_{\rm div}(M) )$ satisfying 
\begin{equation}
\label{mechanical_interpolation}
\left\langle \mathcal{I}\left(x^{(k)}, \dot{x}^{(k)} \right), 
\eta^{(k)}_{x^{(0)}} \right\rangle_\alpha = 0, \quad 
\text{for all} \quad \eta^{(k)}_{x^{(0)}} \in 
\mathfrak{g}^{(k)}_{x^{(0)}}\,,
\end{equation}
where $\left\langle\cdot ,\cdot \right\rangle_\alpha$ is a
weak inner product depending on a parameter $\alpha>0$ which
is a perturbation of the usual $L^2$ inner product on 
$\mathfrak{X}_{\rm div}(M)$. For example this can be the
$H^1$-$\alpha$-metric used in $\alpha$-models and the
Camassa-Holm equation (\cite{HoMaRa1998,foias2001navier}).
We call such an interpolation method $\mathcal{I}$ a \emph{mechanical interpolation method}.
In such a case, if $L$ is a classical Lagrangian, the
difference between the kinetic energy and a potential, the reduced Lagrangian takes the form
\[
\ell^{(k)}_\alpha\left(x^{(k)},\dot{x}^{(k)}, 
\xi^{(k)}_{x^{(k)}}\right) = \frac{1}{2} 
\left\| \inter(\dot{x}^{(k)}) \right\|^2_\alpha + 
\frac{1}{2} \left\| \xi^{(k)}_{x^{(k)}} \right\|_\alpha^2 
- U\left(x^{(k)}\right).
\]
Moreover, we find that
$\Delta_{\mathcal{I}}$ is invariant under the Euler-Lagrange
 flow as a result of Noether's theorem.

\begin{lem} \label{lem:Noether}
Assume the setup of Theorem \ref{thm:particle_method} and that $L$ is the difference between a kinetic minus a potential energy with respect to a right-invariant weak
Riemannian metric $\lb \cdot , \cdot \rb_\alpha : T \SDiff(M) \oplus T \SDiff(M) \to \mathbb{R}$.
Moreover, assume there exists a $k^{\rm th}$ order interpolation method which is mechanical with respect to the metric i.e., \eqref{mechanical_interpolation} holds.
Then the Lagrangian momentum map $J_L : T \SDiff(M) \to \left(\mathfrak{g}^{(k)}_\odot\right)^{\ast}$ is given by
\[
\left\langle J_L( \varphi, \dot{\varphi}) , \eta \right\rangle = \left\langle\dot{\varphi} \circ \varphi^{-1} - \mathcal{I}( j^k_\odot( \dot{\varphi} ) ) , \varphi_*\eta 
\right\rangle_{\alpha}\,, \quad  \text{for all} \quad \eta \in \mathfrak{g}^{(k)}_\odot
\]
and $\Delta_{\mathcal{I}} = J_L^{-1}(0)$.  By Noether's theorem $\Delta_{\mathcal{I}}$ is an invariant submanifold with respect to the flow of the Euler-Lagrange equations on $\SDiff(M)$.
\end{lem}

\begin{proof}
  The right infinitesimal generator of $\eta \in \mathfrak{g}^{(k)}_\odot$ on $\SDiff(M)$ is given by $\varphi \in \SDiff(M) \mapsto T \varphi \circ  \eta$.  Using the definition of the Lagrangian momentum map (see, e.g., 
\cite[Corollary 4.2.14]{FOM}) and the right invariance of $\lb \cdot , \cdot \rb_{\alpha}$ we find
\begin{align*}
	\lb J_L( \varphi, \dot{\varphi} ) , \eta \rb = \lb \dot{\varphi} , T \varphi \circ  \eta \rb_{\alpha}
	&= \lb \dot{\varphi} \circ \varphi^{-1} , \varphi_* \eta \rb_{\alpha}.
\end{align*}
By assumption, $\mathcal{I}$ maps $T_{x^{(k)}}X^{(k)}$ to
the orthogonal of $\mathfrak{g}^{(k)}_{x^{(0)}}$, where
$x^{(k)} = j^k_\odot( \varphi)$ and $x^{(0)} = \pi^{(k)}_0\left(x^{(k)}\right)$. Since $\varphi_*\eta \in 
\mathfrak{g}^{(k)}_{x^{(0)}}$, the formula above implies
\[
\left\langle J_L( \varphi, \dot{\varphi} ) , 
\eta \right\rangle = 
\left\langle \dot{\varphi} \circ \varphi^{-1} - 
\mathcal{I}( j^k_\odot( \dot{\varphi}) ), \varphi_* \eta \right\rangle_{\alpha}.
\]
Thus, by inspection, $\Delta_{\mathcal{I}} = J_L^{-1}( 0 )$.
By Noether's theorem, $J_L$ is conserved under the flow of the Euler-Lagrange equations.
Thus if $\varphi_t$ satisfies the Euler-Lagrange equations and $\dot{\varphi}_0 \in \Delta_{\mathcal{I}}$, then $\dot{\varphi}_t \in \Delta_{\mathcal{I}}$ for all time.
\end{proof}

In other words, given a mechanical interpolation method, the constraint force to keep trajectories in $\Delta_{\mathcal{I}}$ vanishes.
In particular, Lemma \ref{lem:Noether} provides a new version of Theorem \ref{thm:particle_method}.

\begin{cor}\label{cor:particle_method}
Assume the setup of Theorem \ref{thm:particle_method} and Lemma \ref{lem:Noether}.
The reduced Lagrangian takes the form
\begin{align}
\ell^{(k)}\left(x^{(k)},\dot{x}^{(k)}, \xi^{(k)}_{x^{(k)}}\right) = \frac{1}{2} \left\| \inter\left(\dot{x}^{(k)} \right) \right\|_{X^{(k)}}^2 + \frac{1}{2} 
\left\| \xi^{(k)}_{x^{(k)}} \right\|^2_{
\tilde{ \mathfrak{g} }_{\odot}^{(k)} } - 
U\left(x^{(k)} \right). 
\label{eq:sum_of_squares}
\end{align}
Finally, let $L_{X^{(k)}} = 
\left. \ell^{(k)} \right|_{TX^{(k)}}$.  Then the external force $F$ of \eqref{eq:force} is zero and the following are equivalent:
\begin{itemize}
\item[{\rm (i)}] $\dot{\varphi}_0 \in \Delta_{\mathcal{I}}$ and $\varphi_t$ satisfies the 
Euler-Lagrange equation
\[
\frac{D}{Dt} \left(\pder{L}{\dot{\varphi}_t } \right) - \pder{L}{\varphi_t} = 0;
\]
\item[{\rm (ii)}] $\xi^{(k)}_{x^{(k)}}(t) = 0$ and 
$x^{(k)}(t)$ satisfies the Euler-Lagrange equations
\[
\frac{D}{Dt} \left( \pder{L_{X^{(k)}}}{\dot{x}^{(k)}}\right) - \pder{L_{X^{(k)}} }{x^{(k)}} = 0;
	\]
\item[{\rm (iii)}] $\dot{\varphi}_0 \in \Delta_{\inter}$ 
and the path $\varphi_t$ is critical for the action
\[
S[ \varphi] =  \int_0^1{ L(\varphi, \dot{\varphi} ) dt} 
\]	
with respect to variations $\delta \varphi_t \in T\SDiff(M)$ which vanish at the endpoints;
\item[{\rm (iv)}] $\xi^{(k)}_{x^{(k)}}(t) = 0$ and the path
$x^{(k)}(t)$ is critical for the action
\[
S \left[x^{(k)} \right] = 
\int_0^1{ L_{X^{(k)}}\left(x^{(k)},\dot{x}^{(k)}\right) dt} = 0
\]
with respect to arbitrary variations $\delta x^{(k)}$ 
with fixed endpoints.
\end{itemize}
\end{cor}

\begin{proof}
As $\ell^{(k)}$ is quadratic in $\xi^{(k)}_{x^{(k)}}$, 
we conclude that the fiber derivative 
$\partial \ell^{(k)}/\partial \xi^{(k)}_{x^{(k)}}$ at
$\xi^{(k)}_{x^{(k)}}= 0$ vanishes.  Thus $F = 0$.
Moreover, by Lemma \ref{lem:Noether}, we know that the constraint $\dot{\varphi} \in \Delta_{\mathcal{I}}$ is redundant when $\dot{\varphi}_0 \in \Delta_{\mathcal{I}}$.  
The equivalence of items (i) through (iv) is a re-statement of Theorem \ref{thm:particle_method} with $F=0$ and this redundancy of the constraints.
\end{proof}

\begin{remark}{\rm 
  We thank Paula Baseiro for pointing out the following observation.  Corollary \ref{cor:particle_method} can be seen as a Lagrangian mechanical interpretation of the Hamiltonization methods used in non-holonomic systems.  In particular, Corrollary \ref{cor:particle_method} deals with a \emph{Chaplygin system} and can be seen as a special case of \cite[theorem 8.4]{Koiller1992}.} \quad $\lozenge$
\end{remark}

Corollary \ref{cor:particle_method} suggests that we may estimate solutions to certain Lagrangian systems on $\SDiff(M)$ by solving an Euler-Lagrange equation on $X^{(k)}$.
Then the estimated velocity field would be given by $\mathcal{I}\left( \dot{x}^{(k)}\right)$.
Such methods exhibit potentially desirable properties, which we clarify now.

\subsection{Kelvin's circulation theorem for particles} \label{sec:kelvin}
Recall the remaining symmetry of the horizontal equation mentioned in Remark \ref{rmk:higher_order}.  If the 
$k^{\rm th}$
order interpolation method is compatible with this symmetry,
then the numerical method suggested in Corollary
\ref{cor:particle_method}(ii) also possesses this symmetry.
This is a valuable property if one searches for numerical
methods conserving geometric features of the system.  We formalize this statement in the following theorem.

\begin{thm} \label{thm:discrete_kelvin}
Assume the setup of Corollary \ref{cor:particle_method}.  
Let $x^{(k)}(t) \in X^{(k)}$ be an integral curve of the 
Euler-Lagrange equations for $L_{X^{(k)}}$.  If the interpolation method, $\inter$, is 
$\mathcal{J}^k_\odot(G_\odot)$-invariant, then 
$L_{X^{(k)}}$ is $\mathcal{J}^k_{\odot}(G_\odot)$-invariant.  Moreover, for each $\omega$ in the Lie algebra of $\mathcal{J}^k_{\odot}(G_\odot)$, the quantity
\[
J_\omega\left(\dot{x}^{(k)}\right) := 
\left\langle \dot{x}^{(k)} , x^{(k)} \circ \omega 
\right\rangle_X
\]
is conserved.  Here, $x \circ \omega \in  
T_x X^{(k)}$ is given by $x \circ \omega:= 
\left. \frac{d}{d\epsilon} \right|_{\epsilon = 0} 
( x \circ y_\epsilon)$ for an arbitrary curve 
$y_\epsilon$, such that 
$\left.\frac{d y_\epsilon}{d \epsilon}\right|_{\epsilon=0} =
\omega$ (composition is for jets).  Finally, $J_\omega( T\pi( \dot{\varphi}_t) )$ is a conserved quantity for the Euler-Lagrange equations on $\SDiff(M)$, where 
$j^k_{\odot}: \SDiff(M) \to X^{(k)}$ is the quotient projection. If $\dot{\varphi}_0 \in \Delta_\mathcal{I}$ and 
$\dot{x}^{(k)}_0 = j^{(k)}_\odot 
\left(\dot{\varphi}_0\right)$, 
then $J_\omega\left(\dot{x}^{(k)}\right) = 
J_\omega\left(j^{(k)}_\odot \left(\dot{\varphi}\right)
\right)$.
\end{thm}

\begin{proof}
If $\inter( \cdot)$ is 
$\mathcal{J}_\odot^k(G_\odot)$-invariant, then 
$\inter(\dot{x} \circ y) = \inter(\dot{x})$ for 
all $y \in \mathcal{J}_\odot^{k}(G_\odot)$.  Thus 
$L_{X^{(k)}}( \dot{x} \circ y) = L( \inter(\dot{x} \circ y) ) 
= L( \inter( \dot{x}) ) = L_{X^{(k)}}(\dot{x})$, and so 
$L_{X^{(k)}}$ is $\mathcal{J}_{\odot}^k(G_\odot)$-invariant.  
Moreover, if $L$ is $G_\odot$-invariant, then $\ell^{(k)}$ 
is $G_\odot / G^{(k)}_\odot$ invariant.  Remembering the 
identification $\mathcal{J}^k_\odot(G_\odot) \equiv 
G_\odot / G_\odot^{(k)}$, the desired invariance holds for 
$\ell^{(k)}$. By Noether's Theorem, $J_\omega(\dot{x})$ and 
$J_\omega( T\pi(\dot{\varphi}))$ are conserved along the
flow of the Euler-Lagrange equations for $L$ and 
$L_{X^{(k)}}$, respectively. If,
in addition $\dot{\varphi}_0 \in \Delta_\mathcal{I}$ and 
$\dot{x}^{(k)}_0 = j^{(k)}_\odot 
\left(\dot{\varphi}_0\right)$, then $\dot{x}^{(k)} = 
j^{(k)}_\odot \left(\dot{\varphi}\right)$ by 
Corollary \ref{cor:particle_method} and the theorem 
follows.
\end{proof}

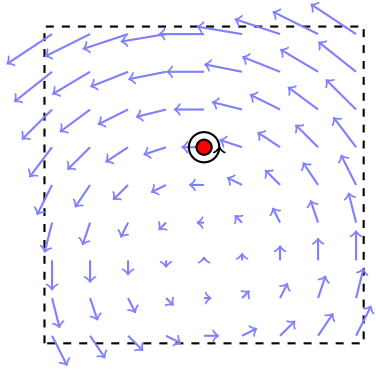
\begin{figure}[h]
\centering
\begin{tikzpicture}[thick]
\draw[dashed] (-2.1,-2.1) rectangle (2.1,2.1);
\foreach \x in {-2.0,-1.5,...,2.0}
\foreach \y in {-2.0,-1.5,...,2.0}
\draw[->,draw=blue!50] (\x,\y) -- ({\x-0.2*(\y+1)} ,
{0.2*\x+\y});
\draw[fill=red] (0,0.5) circle (0.1cm);
\draw (0,0.5) circle (0.2);
\draw[->] (0.2,0.5) -- (0.2,0.51);
\end{tikzpicture}
\caption{A particle-centric version of Kelvin's circulation
theorem corresponds to conservation of circulation along
infinitesimal curves circling the particles.  For $k=1$ this implies
that the momenta conjugate to the spin and the rate of stretching are conserved.}
\label{fig:circulation}
\end{figure}

Theorem \ref{thm:discrete_kelvin} may be viewed as a particle-centric version of Kelvin's circulation theorem.
In particular, it is known that the conserved quantity for the particle relabeling symmetry of an ideal fluid is given by the circulation.
This implies that the integral of the dot product of the fluid velocity along the tangent vector of an advected closed curve is conserved by the
inviscid fluid equations~\cite[Chapter 1]{ArKh1992}.
Theorem \ref{thm:discrete_kelvin} tells us that infinitesimal loops around the particles (see Figure \ref{fig:circulation}) conserve the same quantity and
this conservation law can be manifested in a higher order particle method.
If $k=1$, this implies that the flow conserves the momenta conjugate to the `spin' and `strain rate' of the particles.
For $k > 1$, we can interpret the conserved quantities as manifestations of Kelvin's circulation for $k^{\rm th}$ order perturbations of these infinitesimal loops.
Therefore, when modeling the fluid using the \emph{finite dimensional} Lagrangian system on $X^{(k)}$ with Lagrangian $L_{X^{(k)}}$, one can better control the conservation of momenta and Kelvin's circulation theorem.  In principle, this manifestation of Noether's theorem should allow one to construct integrators that conserve momenta in discrete time; see~\cite{MaWe2001} or~\cite{Hairer2002} for details.

\begin{remark}{\rm 
Variational particle methods for the EPDiff equation have been implemented and analyzed in~\cite{ChDuMa2012} using the
reconstruction mapping mentioned in Remark \ref{ex:interpolation_method}.
The convergence of this method was used to prove global existence and uniqueness in~\cite{ChLiPe2012}.
Therefore, these ideas are not without precedent when $k=0$.}
\end{remark}

\subsection{Spectral Methods and Hybrid Spectral-Particle Methods}
Although the map $\inter$ generally depends on the locations of each of the particles, this may not always be true.  Building upon the ideas of the previous section, we can consider a method which is independent of particle locations
for the case $k=0$.
To do this, choose $N$ linearly independent vector fields, $u_1, \dots, u_N \in \mathfrak{X}_{\mathrm{div}}(M)$.
Then construct the interpolation method $\inter: TX \to
\mathfrak{X}_{\mathrm{div}}(M)$ given by
\begin{align}
\inter(\dot{x}) = \sum_{i=1}^{N} { c_i(\dot{x}) u_i }, \label{eq:spectral}
\end{align}
where the coefficients $c_i$ are described implicitly by the
constraints of Definition \ref{def:kth_order_inter} (it is a linear
algebraic inverse problem). 
The image of the interpolation method given by \eqref{eq:spectral}  on a tangent fiber $T_x X$ is independent of the base point $x \in X$.  In other words, for each vector field given by $\mathcal{I}(x,\dot{x})$ and any point 
$y \in X$ we can find a $\dot{y} \in T_y X$ such 
that $\mathcal{I}(y,\dot{y}) = \mathcal{I}(x,\dot{x})$.  To prove this, consider two distinct particle configurations, $x,y\in X$.  Let $\dot{x} \in T_x X$ and let $c_1, \dots, c_N \in \mathbb{R}$ satisfy \eqref{eq:spectral}.  As the range of $\mathcal{I}$ on the fiber $T_x X$ is identical to the range of $\mathcal{I}$ on the fiber $T_y X$ (it is always in the
span of $u_1, \ldots, u_N$ by \eqref{eq:spectral}), we 
choose 
$\dot{y} =\left. \mathcal{I}^{-1} \right|_{T_y X} \left(\mathcal{I}(x,\dot{x})\right)$. By construction,
$\inter(x,\dot{x}) = \inter(y,\dot{y})$.  In summary, the interpolation method \eqref{eq:spectral} produces a \emph{spectral method}.  That is to say,  solving the Lagrange-d'Alembert equations of Theorem \ref{thm:particle_method}, boils down to solving for vector 
fields contained in the finite dimensional subspace 
$V := \mathrm{span} \{ u_1, \dots, u_N \}$ of the infinite dimensional vector space $\mathfrak{X}_{\mathrm{div}}(M)$.
\begin{remark}{\rm 
One should probably choose $\{u_1,\dots,u_N\}$
such that the space orthogonal to the span of $\{u_1, \ldots, u_N\}$ is a function space of decreasing regularity for increasing $N$,
e.g., a Fourier basis on $M=\mathbb{R}^d$.}
\end{remark}

Judicious choices of bases lead to nice error bounds.
One does not need to choose a full basis, however.
One could choose vector fields $u_1, \dots, u_K$ for $K<N$ to construct an interpolation method of the form
\begin{align}
	\inter(\dot{x}) = \left( \sum_{i=1}^{K}{ c_i u_i } \right) + \left( \sum_{i = K+1}^{N}{ \mathcal{I}_i( x_{K+1}, \dots, x_N , \dot{x}_i ) } \right)  \label{eq:spectral_particle}
\end{align}
which is partially spectral but may depend on the locations 
of the  particles of $x_{K+1}, \dots, x_N$.
The $\mathcal{I}_i$'s are vector bundle morphisms to the vector space $\mathfrak{X}_{\mathrm{div}}(M)$ such that $\inter$ is an interpolation method (this condition does not determine the $\mathcal{I}_i$'s and so one still has some freedom in choosing these maps).
The method induced by \eqref{eq:spectral_particle} can be considered a multiscale model.
We can speculate this claim when $u_1, \dots, u_k$ are chosen to keep track of the `large-scale dynamics' while the remaining degrees of freedom (which depend on the motion of particles) can be interpreted as `fine-scale parameters'.

\subsection{Vortex Methods}
It was shown in~\cite{OlShk2001} that Chorin's vortex blob algorithm (originally developed for simulating inviscid
fluids) provides an exact solution to the equations of motion for an ideal, inviscid, homogeneous, incompressible second grade fluid.
More specifically, Chorin's vortex blob algorithm yields solutions to the Lagrangian system on $\SDiff(M)$ with the reduced Lagrangian given in terms of the spatial velocity by
\begin{align}
	\ell^{(1)}_\alpha(u) := \frac{1}{2} \int_{\mathbb{R}^2}{ u(m) \cdot ( [1- \alpha^2 \Delta] \cdot u)(m) dm}, \label{eq:alpha_lag}
\end{align}
for some $\alpha > 0$.
By construction, the vortex blob method conserves circulation by advecting Gaussian blobs of vorticity.
In this subsection, we will describe in words how the vortex blob method may be viewed as a particular case of the algorithm described in this paper.
Consider a $1^{\rm st}$ order interpolation method, $\inter$, such that the spin of each particle is mapped by $\inter$ to the spatial velocity field of a vortex blob, and the translational component of the particles is mapped by $\inter$ to a spatial velocity field with zero vorticity above each particle.  Then, by Theorem \ref{thm:discrete_kelvin}, if we initialize the system with a spatial velocity field consisting solely of vortex blobs located above the particles, it follows that $\inter(\dot{x})$ remains a sum of vortex blobs for all time.
Additionally, the particles are advected by these vortex blobs and the vorticities extremize the action of 
$\ell^{(1)}_\alpha$.
This is precisely the evolution of the vortex blob algorithm.

  Additionally, there are variants of the vortex blob method which allow for deformations of vortex blobs, as would be the case of an algorithm induced by a $k^{\rm th}$ order interpolation method for $k > 1$. For example,~\cite{UmWaBa2010} used Hermite functions to generate a basis of vorticity functions above each vortex blob to do just this.
It is not immediately clear whether the method of \cite{UmWaBa2010} satisfies a variational principle and so it is not known whether such a method matches the framework described in this paper.
Nonetheless, appealing to higher-order deformations of vortices did yield higher accuracy in numerical experiments for the case $k=2$.
Another interesting method potentially linked to our approach, is the adaptive smoothed particle hydrodynamics proposed in~\cite{ASPH} wherein
the particles are equipped with shape data.

\section{Extensions to other fluids} \label{sec:other_fluids}

Given the structure presented so far, we may summarize the process of extending these ideas to other types of fluids.  In this section, we  sketch constructions to apply this framework to complex fluids, Euler-$\alpha$ models, and template matching problems.

\subsection{Complex Fluids}

Consider the set $\mathcal{F}(M,\mathcal{O})$ of smooth maps from a manifold $M$ to a Lie group $\mathcal{O}$.
The Lie group $\SDiff(M)$ acts on $\mathcal{F}(M,\mathcal{O})$ on the right by group homomorphisms via $f \in \mathcal{F}(M,
\mathcal{O}) \mapsto f \circ \varphi \in \mathcal{F}(M,\mathcal{O})$ for each $\varphi \in \SDiff(M)$.
This right action defines the semi-direct product Lie group $\SDiff(M) \,\circledS\, \mathcal{F}(M,\mathcal{O})$, where the group multiplication is $(\varphi_1,f_1) \cdot (\varphi_2,f_2):= (\varphi_1 \circ \varphi_2 , (f_1 \circ \varphi_2) \cdot f_2)$.
A unifying framework for ideal complex fluids was described in~\cite{GR2009} by performing Euler-Poincar\'{e} reduction on Lagrangian systems with configuration manifolds of the form $\SDiff(M)\,\circledS\, \mathcal{F}(M,\mathcal{O})$.
The number of scenarios captured by this framework is truly remarkable (magnetohydrodynamics, spin-glass fluids, liquid crystals, fluids with microstructure, Yang-Mills fluids,
superfluids, etc.).
As with ideal fluids, these systems exhibit $\SDiff(M)$ symmetry, and so it should be clear that one may reduce by $G_\odot^{(k)} \subset \SDiff(M)$.
This could be particularly fruitful for the case of fluids with microstructure, where the Lagrangian depends on the relative orientations of fluid particles.
For example, in the case of micromorphic nematic liquid crystals, $\mathcal{O} = \SO(3)$.
We see that $X^{(1)}$ for $M =\mathbb{R}^3$ is the trivial
bundle $(\mathbb{R}^3)^N \times \SL(3, \mathbb{R})^N$. Since 
$\SO(3)$ is a subgroup of $\SL(3,\mathbb{R})$, it follows that 
a method on $X^{(1)}$ is capable of expressing the orientation of liquid crystals.  However, the Lagrangian for liquid crystals usually depends on the gradient of a field in $\mathcal{F}(M, \mathcal{O})$.  Thus, we must be able to express gradients of diffeomorphisms above each particle, and this suggests reducing by $G_\odot^{(2)}$ and using Euler-Lagrange equations on $X^{(2)}$ to model liquid crystals.

\subsection{Second grade fluids and template matching problems}

A kinetic energy on $M = \mathbb{R}^3$ given by \eqref{eq:alpha_lag} with $\alpha > 0$ leads to the Euler-$\alpha$ model in the incompressible case.
We may even consider the Lagrangian
\[
L_{\mathbb{I}}(\varphi, \dot{\varphi}) = \frac{1}{2}\int{u\cdot\mathbb{I}(u)d^3x},
\]
where $\mathbb{I}$ is a positive definite differential operator on $\mathfrak{X}(M)$ (again, the Helmoltz operator $I - \alpha  \Delta$ is a good example).
The momentum conjugate to $u$ is given by the convolution 
$G * u$, where $G$ is the Green's function for $\mathbb{I}$.
In this case, we may even consider compressible fluids if we use interpolation methods which map to the full set of smooth vector fields on $M$ (see~\cite[Chapter 11]{HoScSt2009} for a good overview).
Moreover, if $G$ is non-singular, the interpolation methods mentioned in Example \ref{ex:interpolation_method} exhibit a mechanical connection and allows us to use Corollary \ref{cor:particle_method} to produce particle methods.  Such methods could be useful in the case of template matching problems, wherein one seeks to find a distance between images by integrating geodesic flows on $\Diff(M)$ (see, 
\cite{Beg2005,Bruveris2011}).  In such a scenario, the material representation of the ``fluid'' is particularly meaningful, and particle methods become a natural choice
(see \cite{BruverisEllis2011,BaBruCoMi2012}).
In particular, the jet-data framework was discussed for arbitrary $k$ in $\mathbb{R}^2$ and was even implemented for $k=1$ \cite{Sommer2013}.

\section{Conclusions}

Since the publication of~\cite{Arnold1966}, it has been well known that Euler's equations of motion for an ideal, inviscid, incompressible, homogeneous fluid may be regarded as Euler-Poincar\'{e} equations on the Lie algebra of divergence free vector fields tangent to the boundary.
Therefore, it has been clear that one could reduce by various subgroups of the diffeomorphism group.
In this paper we have accomplished one version of this reduction, using the notion of an interpolation method.
We chose to use interpolation methods instead of principal connections (as in~\cite{CeMaRa2001}) to clarify the notion of estimating the velocity field of
the fluid with \emph{particles}.
Specifically, if an interpolation method is chosen, one may estimate the spatial velocity field of the fluid and even estimate the evolution of the system over short times by integrating a non-holonomically constrained version of the equations.
We also discussed the reduction process for higher-order isotropy groups.
This discussion led to more sophisticated particle methods which contained extra symmetry and a particle-centric version of Kelvin's circulation theorem.
The particle-centric Kelvin's theorem was particularly meaningful, since the integrals of motion are conserved by the exact evolution on $\SDiff(M)$.

Future work includes:
\begin{itemize}
\item Implementing an instance of the particle method
described in Theorem \ref{thm:particle_method} and Corollary \ref{cor:particle_method} for $k > 1$ and comparing
performance with the vortex blob method and a smooth-particle hydrodynamics method.
\item Extending these constructions to the fluids and applications mentioned in section \S \ref{sec:other_fluids}.  In particular, we would like to consider applications to liquid crystals and image matching problems.
\item Appending a finite dimensional model of the vertical Lagrange-Poincar\'{e} equations would yield a finite dimensional model of a fluid on a Lie groupoid.
This generalizes ideas contained in ~\cite{Gawlik2011} wherein a finite dimensional Lie group is used.
This could be viewed as a turbulence model in the spirit of~\cite{HoTr2012}.
\end{itemize}

\noindent\textbf{Acknowledgments.}
During the work on this paper we have gotten advice and guidance from many people in a range of fields.
In particular, we would like to thank Andrea Bertozzi and David Uminsky for listening to our ideas and educating us on state of the art vortex blob methods.
We also thank Chris Anderson, Marcel Oliver, and Melvin Leok for helping us avoid certain common pitfalls which occur when the theoretically minded venture into numerics.
We owe special thanks to Darryl D. Holm, Hui Sun, and Joris Vankerschaver for thoughtful discussions and suggestions.
Finally, we thank Jerrold E. Marsden who spent the final years of his life advising the first author with much thought and care.
H. Jacobs and M. Desbrun were funded by NSF grant CCF-1011944.
T.S. Ratiu was partially supported by the government grant of the Russian Federation for support of research projects implemented by leading scientists, Lomonosov Moscow State University under the agreement No. 11.G34.31.0054 and Swiss NSF grant 200021-140238.

\begin{appendix}

\section{A coordinate free derivation of the Lin constraints} \label{app:Lin}
Let $G$ be a Lie group with Lie algebra $\mathfrak{g}$.  We will adopt the standard convention wherein the Lie bracket of $\mathfrak{g}$ is taken to be the Jacobi-Lie bracket of \emph{left} invariant vector fields.  Let $\rho \in \Omega^1( G ; \mathfrak{g})$ be the (right) Maurer-Cartan form
\begin{equation}
\label{MC_form_general}
	\rho( v_g ) = TR_g^{-1} \cdot v_g, \quad v_g \in T_gG.
\end{equation}
By the definition of the exterior-derivative we see that for any $X, Y \in \mathfrak{X}(G)$ we have
\[
	d\rho( X , Y) = \boldsymbol{\pounds}_X [ \rho ( Y ) ] - \boldsymbol{\pounds}_Y[ \rho ( X ) ] - \rho \left( [ X , Y ] \right),
\]
where $\boldsymbol{\pounds}_X$ is the Lie derivative defined
by the vector field $X$.
If $u_g, v_g \in T_g G$, define $\xi := \rho(u_g)$, $\eta := \rho(v_g)$ and construct the right-invariant vector fields 
$\xi_R, \eta_R \in \mathfrak{X}(G)$ defined by $\xi_R(h) = \xi \cdot h$ and $\eta_R(h) = \eta \cdot h$ for any $h \in G$.  Using this we find
\begin{align*}
d\rho(u_g , v_g) &= d\rho( \xi_R(g) , \eta_R(g) ) 
= d\rho( \xi_R , \eta_R) (g) \\
&=  \left\{\boldsymbol{\pounds}_{\xi_R} [ \rho ( \eta_R ) ] - \boldsymbol{\pounds}_{\eta_R}[ \rho ( \xi_R ) ] - \rho \left( [ \xi_R , \eta_R ] \right) \right\}(g) \\
&= -\rho( -[\xi,\eta]_R )(g) = [ \xi , \eta] \\
		&= [\rho(u_g) , \rho(v_g) ].
\end{align*}
This is the \emph{Maurer-Cartan equation}.

\begin{prop}
\label{prop_appendix_A}
  Let $U \subset \mathbb{R}^2$ be an open set and let $g:U \to G$ be an embedding.  If we coordinatize $U$ by $(s,t)$ and let $\partial_s$ and $\partial_t$ denote the corresponding vector fields in $\mathfrak{X}(U)$, define $\dot{g} := Tg \circ \partial_t$ and $\delta g: = Tg \circ \partial_s$.  Finally, if $\xi := \rho \circ \dot{g}$ and $\eta := \rho \circ \delta g$, then
  \[
  	\partial_s \xi = \partial_t \eta + [ \eta , \xi ]
  \]
\end{prop}

\begin{proof}
	As $g:U \to G$ is an embedding, we see that $g(U)$ is a surface in $G$.  Moreover, $\dot{g}(s,t)$ is a vector over $g(s,t)$ tangent to the surface $g(U)$.  The same holds for $\delta g(s,t)$.  Let $X , Y \in \mathfrak{X}(G)$ be arbitrary vector fields which satisfy $X( g(s,t) ) = \dot{g}(s,t)$ and $Y(g(s,t)) = \delta g(s,t)$ for any $(s,t) \in U$.  Then $g_*( \partial_t), g_*( \partial_s) \in \mathfrak{X}( g(U) )$ and $X,Y \in \mathfrak{X}(G)$ satisfy $\left. X \right|_{g(U)} = g_*( \partial_t)$ and $\left. Y \right|_{g(U)} = g_*( \partial_s)$.
	
By the Maurer-Cartan equation we see that
\[
d\rho(\dot{g}(s,t) ,\delta g(s,t) ) = [\xi(s,t) , \eta(s,t) ].
\]
	However, by the definition of the exterior derivative we see that
	\begin{align*}
		d\rho( \dot{g}(s,t) , \delta g(s,t) ) &= d\rho( X , Y)( g(s,t) ) \\
			&= \left\{\boldsymbol{\pounds}_X[ \rho(Y) ] - \boldsymbol{\pounds}_Y[ \rho(X)] - \rho( [X,Y] ) \right\} (g(s,t) ).
	\end{align*}
	Let us look at each of the terms of this sum individually.  We find
\[
\boldsymbol{\pounds}_X[ \rho(Y) ] (g(s,t) ) = \left. \frac{d}{d \tau } \right|_{\tau = 0} \rho(Y( \Phi_X^\tau( g(s,t) ) ) )
\]
	where $\Phi_X^\tau$ is the flow of $X$.  As $X$ is tangential to the surface $g(U)$ we see that $\Phi_X^\tau( g(s,t) ) \in g(U)$.  In particular $\Phi_X^\tau(g(s,t) ) = g(s, t + \tau)$ by uniqueness of integral curves of $X$.  Thus $\rho ( Y( \Phi_X^\tau( g(s,t) ) ) ) = \rho( Y( g(s, t + \tau) ) )= \rho( \delta g(s, t + \tau) ) = \eta(s,t + \tau)$.  In conclusion
\[
\boldsymbol{\pounds}_X[ \rho(Y) ] (g(s,t) ) = \partial_t \eta(s,t).
\]	
By the same token
\[
\boldsymbol{\pounds}_Y[ \rho(X) ](g(s,t) ) = \partial_s \xi.
\]
	Finally, as $g:U \to G$ is an embedding, it preserves the Jacobi-Lie bracket of vector field from $\mathfrak{X}(U)$ to $\mathfrak{X}(g(U))$.  Thus  we find
	\[
		\left. [X, Y] \right|_{g(U)} = [ g_*( \partial_t) , g_*( \partial_s) ] = g_* [ \partial_t, \partial_s] = 0
	\]
and hence
	\[
		\rho( [X , Y] )(g(s,t) ) = \rho( [X,Y](g(s,t) ) ) = \rho(0) = 0.
	\]
	Therefore $[\xi,\eta] = d\rho( \dot{g} , \delta g) = \partial_t \eta - \partial_s \xi$.
\end{proof}

In the applications of this paper, $g(s, t)$ appears as a 
deformation of a given smooth curve $g(t)$. We shall always
use these deformations to be at least immersions, which are
locally embeddings around a point in $G$.

As a final note we remind the reader that the bracket above is with respect to the left-Lie algebra.  However, the relation between the left Lie algebra and the right Lie algebra is given by $[\xi , \eta]_L = - [\xi, \eta]_R$.  Therefore, if one works in the right Lie algebra, then the above proposition states $\partial_s \xi = \partial_t \eta - [\eta , \xi]_{R}$.

In the applications used in this paper, the group is 
$\SDiff(M)$ whose \textit{left} Lie algebra is 
$\mathfrak{X}_{\rm div}(M)$ endowed with minus the standard
Jacobi-Lie bracket of vector fields. So, if $\varphi:U 
\subset \mathbb{R}^2 \rightarrow \SDiff(M)$ is an 
embedding and $\dot{\varphi} = T \varphi \circ \partial_t$,
$\delta \varphi = T \varphi \circ \partial_s$,
$u(s,t) = \dot{\varphi}(s,t) \circ \varphi(s,t)^{-1}$, and
$v(s,t) = \delta\varphi(s,t) \circ \varphi(s,t)^{-1}$, we have
\begin{equation}
\label{lin}
\frac{\partial}{ \partial s} u = \frac{\partial}{ \partial t}v
-[v,u]_{\rm Jacobi-Lie}.
\end{equation}
\end{appendix}

\bibliographystyle{siam}
\bibliography{JaRaDe_2012}

\end{document}